\documentclass[oneside, reqno, 10pt]{amsart}

\usepackage{geometry}

\usepackage[all]{xy}
\usepackage{stmaryrd}
\usepackage[usenames, dvipsnames]{color}
\usepackage[pdfborder={0 0 0}, colorlinks = true, linkcolor=blue, citecolor=OliveGreen]{hyperref}
\usepackage{graphicx}
\usepackage{amssymb}
\usepackage{amsthm}
\usepackage{amsmath, amscd}
\usepackage{mathrsfs}
\usepackage{paralist}
\usepackage{framed}

\DeclareMathOperator{\Q}{\mathbb{Q}}
\DeclareMathOperator{\Z}{\mathbb{Z}}
\DeclareMathOperator{\C}{\mathbb{C}}

\DeclareMathOperator{\R}{\mathbb{R}}
\DeclareMathOperator{\F}{\mathbb{F}}

\DeclareMathOperator{\D}{\mathcal{D}}
\DeclareMathOperator{\length}{length}
\DeclareMathOperator\Aut{Aut}
\DeclareMathOperator\End{End}
\DeclareMathOperator\Hom{Hom}
\DeclareMathOperator\Spec{Spec}

\newcommand{\la}{\langle} 	
\newcommand{\ra}{\rangle} 	
\newcommand{ \lie}[1]{\mathfrak{#1} } 
\newcommand{\isomto}{\overset{\sim}{\longrightarrow}}

\newcommand{ \refThm}[1]{\hyperref[#1]{Theorem \ref{#1}}}
\newcommand{ \refSec}[1]{\hyperref[#1]{Section \ref{#1}}}
\newcommand{ \refProp}[1]{\hyperref[#1]{Proposition \ref{#1}}}
\newcommand{ \refLemma}[1]{\hyperref[#1]{Lemma \ref{#1}}}
\newcommand{\refRemark}[1]{\hyperref[#1]{Remark \ref{#1}}}
\newcommand{\refDef}[1]{\hyperref[#1]{Definition \ref{#1}}}

\theoremstyle{plain}
\newtheorem{theorem}{Theorem}[section]
\newtheorem*{theorem*}{Theorem}
\newtheorem*{maintheorem*}{Main Theorem}

\newtheorem{lemma}[theorem]{Lemma}
\newtheorem{proposition}[theorem]{Proposition}
\newtheorem{definition}[theorem]{Definition}

\newtheorem*{claim*}{Claim}

\numberwithin{equation}{section}

\theoremstyle{remark}
\newtheorem{remark}[theorem]{Remark}

\font\caps=cmcsc10                    

\let\bftemp\bfseries
\renewcommand\bfseries{\normalsize\caps}

\usepackage{framed}
\renewenvironment{leftbar}[1][\hsize]
{%
    \MakeFramed{\hsize#1\advance\hsize-\width\FrameRestore}%
}
{\endMakeFramed}
\let\oldtocsection=\tocsection
\let\oldtocsubsection=\tocsubsection
\let\oldtocsubsubsection=\tocsubsubsection
\renewcommand{\tocsection}[2]{\hspace{0.2em}\oldtocsection{#1}{#2}}
\renewcommand{\tocsubsection}[2]{ \hspace{1em}\oldtocsubsection{#1}{#2} }
\renewcommand{\tocsubsubsection}[2]{\hspace{2em}\oldtocsubsubsection{#1}{#2}}

\DeclareMathOperator{\OB}{\mathcal{O}_B}
\DeclareMathOperator{\OBx}{\mathcal{O}_B^{\times}}
\DeclareMathOperator{\OBp}{\mathcal{O}_{B,p}}
\DeclareMathOperator{\OBxone}{\mathcal{O}_{B}^{\times,1}}

\DeclareMathOperator{\Zed}{\mathcal{Z}}

\DeclareMathOperator{\CB}{\mathcal{C}_{B}}
\newcommand{\CBok}{ \ensuremath{ \mathcal{C}_{B   {\scriptstyle  / o_k }}  }}
\newcommand{\CHR}{\ensuremath{\widehat{CH}{}^1_{\scriptscriptstyle \R} }}

\DeclareMathOperator{\CBZ}{ \mathcal{C}_{B   {\scriptscriptstyle  /  \! \Z }}  }

\title{Unitary cycles on Shimura curves and the Shimura lift II}
\address{Mathematisches Institut der Universit\"at Bonn, Endenicher Allee 60 \\ 
		53115 Bonn, Germany}
\email{sankaran@math.uni-bonn.de}
\author{Siddarth Sankaran}

\begin{document}
\date\today
\begin{abstract}
We consider two families of arithmetic divisors defined on integral models of Shimura curves. The first was studied by Kudla, Rapoport and Yang, who proved that if one assembles these divisors in a formal generating series, one obtains the q-expansion of a modular form of weight 3/2. The present work concerns the Shimura lift of this modular form: we identify the Shimura lift with a generating series comprised of unitary divisors, which arose in recent work of Kudla and Rapoport regarding cycles on Shimura varieties of unitary type. In the prequel to this paper, the author considered the geometry of the two families of cycles; these results are combined with the archimedian calculations found in this work in order to establish the  theorem. In particular, we obtain new examples of modular generating series whose coefficients lie in arithmetic Chow groups of Shimura varieties.
\end{abstract}
\maketitle

\tableofcontents
\section{Introduction}

In this paper, which is the `global' counterpart to \cite{San1}, we examine the relationship between two families of arithmetic divisors defined on integral models of Shimura curves. The first of these families, which we refer to as the \emph{orthogonal special cycles}, have been studied extensively by Kudla, Rapoport and Yang. One of their main results is that if one assembles these arithmetic divisors into a formal generating series, one obtains the $q$-expansion of a modular form of weight 3/2; this remarkable fact is in some sense the most fully-realized instantiation of {Kudla's programme}, a broad set of conjectures relating generating series of arithmetic special cycles on Shimura varieties and the Fourier coefficients of modular forms, cf.\ the expository notes \cite{KudlaMSRI} for a survey.  

Our main theorem is a geometric interpretation of the \emph{Shimura lift} of this generating series. More precisely, we identify the $q$-expansion of the Shimura lift with a generating series of \emph{(arithmetic) unitary special cycles}; these latter cycles are analogues of those constructed in more recent work of Kudla and Rapoport  \cite{KRUloc, KRUglob}, who studied the intersection behaviour of such cycles in integral models of Shimura varieties attached to unitary groups of signature $(n,1)$. While the theory for these cycles is not as fully developped as in the ``orthogonal" case above, the results in the present work (which concerns the case $n=1$) contribute to a growing body of evidence relating these unitary generating series to modular forms, cf.\ e.g. \cite{KRUglob,HBY, How2,How3}.

We now give a more precise statement of our main theorem. Begin by fixing a rational indefinite quaternion algebra $B$ that is unramified at 2, and a maximal order $\OB \subset B$.  Let $\CBZ$ denote the Deligne-Mumford  stack (over $\Spec\, \Z$) that parametrizes pairs $\underline A = (A, \iota_A)$; here $A$ is an abelian surface over some base scheme $S$, and 
\[ \iota_A\colon \ \OB \to \End(A) \]
is an $\OB$-action which is assumed to satisfy a `determinant' condition, cf.\ \refDef{shCurveDef} below. This stack is an integral model of the classical Shimura curve attached to $B$. 

The \emph{orthogonal special cycles} are (more or less, cf.\ \refRemark{CohenMacauleyRmk}) defined by the following moduli problem: for an integer $n>0$, we let $\Zed^o(n)$ be the stack that parametrizes diagrams 
\[ \xi \colon A \to A, \]
where $\underline A = (A, \iota_A)$ is a point of $\CB$, and $\xi$ is an $\OB$-linear traceless endomorphism such that $\xi^2 = -n$. We may then view $\Zed^o(n)$ as a divisor on $\CBZ$ via the natural forgetful map, which we denote by the same symbol $\Zed^o(n)$. 
In \cite{KRYbook}, Kudla, Rapoport and Yang equip these cycles with explicit Green functions $Gr^o(n, v)$ depending on a parameter $v \in \R_{>0}$, and extend the construction to all $n \in \Z$ to obtain arithmetic classes
\[ \widehat{\Zed}{}^o(n, v ) \ := \ (\Zed^o(n,v), Gr^o(n,v)) \ \in \ \CHR(\CBZ)  \]
in the first arithmetic Chow group.\footnote{
Here, and throughout this paper, our arithmetic Chow groups are taken with \emph{real} coefficients, cf.\ \refSec{chowSec}.}
 They then assemble these divisors into a formal generating series, and prove the following remarkable theorem:
\begin{theorem}[Theorem A of \cite{KRYbook}] \label{KRYThmIntro}
For $\tau = u + iv \in \lie{H}$ and $ q_{\tau}  = e^{2 \pi i \tau}$, the \emph{orthogonal generating series}
\[ \widehat\Phi{}^o(\tau) \ := \ \sum_{n \in \Z} \widehat{\Zed}{}^o(n, v) \ q_{\tau}^n \ \in \ \CHR(\CBZ)\llbracket q_{\tau}^{\pm 1} \rrbracket \]
is  the $q$-expansion of a modular form of weight 3/2 for the congruence subgroup $\Gamma_0(4 D_B)$ with trivial character; here $D_B$ is the discriminant of $B$.
\end{theorem}
Our construction of the \emph{unitary special cycles}, which we briefly review here, is adapted from \cite{KRUglob}. Suppose $k = \Q(\sqrt{\Delta})$ is an imaginary quadratic field, with $\Delta<0$ a squarefree integer, and let $o_k$ denote the ring of integers of $k$. Denote the non-trivial Galois automorphism of $k$ by $a \mapsto a'$. Here and throughout the entire paper, we assume that:
\begin{compactenum}[(i)]
\item $\Delta$ is even and
\item every prime dividing $D_B$ is inert in $k$.
\end{compactenum}
In particular, the second condition implies the existence of an embedding $\phi \colon o_k \to \OB$, and for the moment, we fix one such embedding. 

Let $\mathcal E^+$ denote the Deligne-Mumford (DM) stack over $\Spec(o_k)$ representing the following moduli problem: for a base scheme $S$ over $o_k$, the $S$-points form the category of tuples
\[ \mathcal E^+(S) \ = \ \{ \underline E = (E, i_E, \lambda_E) \} \]
where $E$ is an elliptic curve over $S$, together with an $o_k$-action $i_E \colon o_k \to End(E)$, and a principal polarization $\lambda_E$ whose corresponding Rosati involution ${}^*$ satisfies
\[ i_E(a)^* \ = \ i_E(a'). \]
We also require that on the Lie algebra $Lie(E)$, the action induced by $i_E$ coincides  with the action induced by the structural morphism $o_k \to \mathcal O_S$. In the notation of \cite{KRUglob}, this stack is referred to as $\mathcal M(1,0)$. 

Similarly, we define $\mathcal E^-$ to be the moduli space of tuples $\underline E = (E, i_E, \lambda_E)$ as above, except now we insist that the action on $Lie(E)$ coincides with the \emph{conjugate} of the action given by the structural morphism $o_k \to \mathcal O_S$.
Finally, we set 
\[ \mathcal E \ := \ \mathcal E^+ \ \coprod \ \mathcal E^- \] 
to be the disjoint union of these two stacks. 

Let $\CBok = \CBZ \times \Spec(o_k)$ denote the base change. Suppose $S$ is some base scheme over $\Spec(o_k)$, and that we are given points $\underline E \in \mathcal E(S)$ and $\underline A \in \CBok(S)$. For any embedding $\phi \colon o_k \to \OB$, we define the space of \emph{special homomorphisms}:
\[ \Hom_{\phi}(\underline E, \underline A) \ := \ \left\{ \lie{y} \in Hom(E, A) \ | \ \lie y \circ i_E(a) \ = \ \iota_A( \phi (a) ) \circ \lie y, \ \ \text{for all } a \in o_k \right\}. \]
As discussed in \refSec{UcyclesSec}, this space comes equipped with an $o_k$-hermitian form $h_{E,A}^{\phi}$.

For $m \in \Z_{>0}$, the \emph{unitary special cycle} $\Zed(m, \phi)$ is  (the DM stack over $o_k$ representing) the following moduli problem: for a scheme $S$ over $o_k$, the $S$-points of $\Zed(m, \phi)$ comprise the category of tuples
\[ \Zed(m, \phi)(S) \ = \{ (\underline E, \underline A, \lie{y})\}\]
where $\underline E \in \mathcal E(S), \underline A \in \CBok(S)$, and $\lie{y} \in \Hom_{\phi}(E, A)$ such that $h_{E, A}^{\phi}(\lie y, \lie y ) = m$. 
Abusing notation, we use the same symbol $\Zed(m, \phi)$ to denote the divisor on $\CBok$ obtained via the natural forgetful morphism $\Zed(m, \phi) \to \CBok$. We equip these divisors with Green functions $Gr^u(m, \phi, \eta)$ that depend on a real parameter $\eta >0$, and whose construction parallels that of the  Green functions $Gr^o(n, v)$ for the orthogonal cycles.  We thereby obtain classes
\[ \widehat\Zed(m, \phi, \eta) \ := \ (\Zed(m, \phi), \ Gr^u(m, \phi, \eta))  \ \in \ \CHR(\CBok). \]
We also describe how to extend the definitions to all integers $m \in \Z$. In addition, we shall require the following rescaled variants: 
\begin{equation*} 
 \widehat\Zed{}^*(m, \phi; \eta) := (\Zed^*(m, \phi) , \  Gr^{u*}(m, \phi, \eta)), 
\end{equation*}
where
\[ \Zed^*(m, \phi):=   \Zed \left( \frac{ m}{(|m|,D_B)}, \ \phi\right), \ \   \text{and } Gr^{u*}(m, \phi, \eta):= Gr^u \left(  \frac{ m}{(|m|,D_B)}, \ \phi, \  (|m|, D_B)^{1/2} \cdot \eta\right). \]
This leads us to the \emph{unitary generating series}: for $w = \xi + i \eta \in \lie{H}$ and $q_w = e^{2 \pi i w}$, we set
\begin{align*} \widehat \Phi^u(w) := \widehat\Zed(0, \eta) \ + \ \frac{1}{4 h(k)} \sum_{m \neq 0} \sum_{[\phi] \in Opt / \OBxone} \ \left( \widehat\Zed(m, \phi,\eta) \ + 
\ \widehat\Zed{}^*\left( m , \phi, \eta \right)  \right) q_w^m,
\end{align*}
where the sum on $[\phi]$ is over a set of equivalence classes of \emph{optimal embeddings} $o_k \to \OB$, cf.\ \eqref{optEmbDefEqn}.

The main theorem of this paper describes the relationship between the orthogonal and unitary generating series in terms of the \emph{Shimura lift}, first introduced by Shimura \cite{Shim} and later extended to the non-holomorphic setting by Shintani and Niwa, cf.\ \cite{Shin,Niwa}. The Shimura lift is an operation which takes as its input a modular form $F$ of half-integral weight $\kappa/2$, together with a square-free parameter $t \in \Z_{>0}$, and produces a modular form $Sh_t(F)$ of even integral weight $\kappa - 1$. 

\begin{maintheorem*}Suppose $\Delta<0$ is an even squarefree  integer, and every prime dividing $D_B$ is inert in $k = \Q(\sqrt{\Delta})$. Let
\[ \widehat{\Phi}^o_{/o_k} (\tau) \ = \ \sum_{n \in \Z} \widehat{\Zed}{}^o(n, v)_{/o_k} \ q^{n}_{\tau} \ \in \ \CHR(\CBok) \llbracket q_{\tau}^{\pm 1} \rrbracket \]
denote the base change of the orthogonal generating series to $\Spec(o_k)$.
Then, for $w = \xi + i \eta \in \lie{H}$, 
\[ Sh_{|\Delta|} \left(  \widehat{\Phi}^o_{/o_k} \right) (w) \ = \ \widehat\Phi{}^u(w). \]
In particular, $\widehat\Phi{}^u$ is (the $q$-expansion of) a modular form of weight 2. 
\qed
\end{maintheorem*}

The proof  hinges on a decomposition
\[ \CHR(\CBok) \  = \ \widetilde{MW} \ \oplus \ \left( Vert \oplus \R \widehat\omega \right) \ \oplus \ An    \]
cf.\ \refSec{chowSec}, which is orthogonal for the Arakelov-Gillet-Soul\'e intersection pairing. We may accordingly decompose our generating series into components, and our strategy for the proof of the main theorem is to show that the Shimura lift formula holds one component at a time. More precisely, we recall that Kudla, Rapoport and Yang prove the modularity of the orthogonal generating series (\refThm{KRYThmIntro}) by showing that each of its components taken individually is the $q$-expansion of a known modular form. Our main theorem amounts to saying that for each component, the Shimura lift of this form can be identified with the corresponding component of $\widehat\Phi{}^u$. 

The $\widetilde{MW}$ and $Vert$ components depend only on the geometry of the cycles, and not the Green functions. This was studied in the prequel to the present work, and in particular the desired relations for these components follows immediately from \cite[Corollary 4.11]{San1}.

The `analytic' component $An$ is spanned by classes of the form $\widehat f = (0, f)$, where $f$ is a smooth function on the complex points of $\CBok$, and can be thought of as being `vertical at $\infty$'. In \refSec{anSec}, we prove the main theorem for these components, which involves computing relations between the two families of Green functions, and explicit formulas derived from Niwa and Shintani's presentation of the Shimura lift via theta series. 

Finally, the desired theorem for the `Hodge components', obtained by pairing against the so-called `Hodge bundle' $\widehat\omega$, is proven in \refSec{hodgeSec}, and largely follows from  considerations similar to the previous case. 
\subsection*{Acknowledgements}
This paper is essentially an updated version of my Ph.\ D.\ thesis. I am grateful to my advisor S. Kudla for his constant encouragement. I would also like to thank B. Howard for his detailed and helpful comments on previous versions of this paper.  This work was supported by the SFB/TR 45 `Periods, Moduli Spaces, and Arithmetic of Algebraic Varieties' of the DFG (German Research Foundation), and by NSERC (National Sciences and Engineering Research Council of Canada).

\subsection*{Notation}  \label{notationSec}
\begin{compactitem}
\item $B$ is a fixed indefinite rational quaternion algebra, with maximal order $\OB$, and discriminant $D_B > 1$. Denote the reduced trace and reduced norm on $B$ by $Nrd$ and $Trd$ respectively. 
\item $\OBxone$ is the subgroup of $\OBx$ consisting of elemenets of reduced norm 1. 
\item $k = \Q(\sqrt{\Delta})$ is an imaginary quadratic field, with ring of integers $o_k$, such that (i) $\Delta < 0$ is a squarefree \emph{even} integer and (ii) every prime dividing $D_B$ is \emph{inert} in   $k$. We denote the non-trivial Galois automorphism on $k$ by $a \mapsto a'$. 
\item Having fixed the generator $\sqrt{\Delta} \in k$, we let $\sigma_0 \colon k \to \C$ be the complex embedding such that $\sigma_0 (\sqrt{\Delta}) \ = \ |\Delta|^{1/2} i$. Let $\sigma_1 = \sigma_0'$ denote the conjugate embedding. 
\item By definition, an embedding $\phi\colon o_k \hookrightarrow \OB$ is called \emph{optimal} if 
\begin{equation} \label{optEmbDefEqn}
\phi(o_k) \ = \ \phi(k) \cap \OB. 
\end{equation}
Note $\OBxone$ acts on the set of optimal embeddings by conjugation, i.e., for $\xi \in \OBxone$, 
\[ (\xi \cdot \phi) (a) \ :=  \ Ad_{\xi} \circ \phi (a) \ = \ \xi \cdot \phi(a) \cdot \xi^{-1}, \qquad \text{for all } a \in o_k.\]
Let $Opt$ denote the set of optimal embeddings, and $Opt / \OBxone$ denote the set of optimal embeddings taken up to $\OBxone$-equivalence. 
\item  $\lie{H} = \left\{ z \in \C, \Im(z) > 0 \right\} $ is the usual complex upper-half plane, and $\lie{H}^{\pm} = \C \setminus \R$ is the union of the upper- and lower- half planes.
\end{compactitem}

\section{Shimura curves and orthogonal special cycles}

We begin this section by reviewing the theory of arithmetic Chow groups that we will use, and then we recall the definitions and basic facts from \cite{KRYbook} regarding Shimura curves and the generating series of orthogonal special cycles.
\subsection{Chow groups for arithmetic surfaces} \label{chowSec} {\ \\}

We briefly recall the construction of the first arithmetic Chow groups for arithmetic surfaces, mainly for the purposes of setting up notation; basic references include \cite{GilletSoule, Lang} for the case of schemes, and \cite{Vistoli, KRYbook} for the necessary modifications in the case of Deligne-Mumford (DM) stacks. 

Let $O_L$ be the ring of integers of a number field. For us, an \emph{arithmetic surface}  over $O_L$ is a regular, geometrically irreducible, Deligne-Mumford stack $\mathcal X$ of dimension 2, together with a proper and flat morphism $\mathcal X \to Spec(O_L)$. Let $\Sigma$ denote the set of embeddings $L \hookrightarrow \C$; for any DM stack $\mathcal Y$ over $O_L$ and $\sigma \in \Sigma$, we let $\mathcal Y(\C^{\sigma})$ denote the complex orbifold attached to the base change $\mathcal Y \times_{\sigma, o_k} Spec(\C).$
\begin{definition} \label{arDivDef}
 An $(\R)$-\emph{arithmetic divisor} on an arithmetic surface $\mathcal X$ is a tuple 
\[  \widehat \Zed = \left( \mathcal Z,   Gr(\mathcal Z) \right) \] where 
\begin{compactenum}[(i)]
\item $\mathcal Z = \sum n_i \mathcal Z_i$ is a formal $\R$-linear combination of pairwise distinct prime divisors $ \Zed_i$, i.e. each $\Zed_i$ is an irreducible closed substack which is \'etale locally a Cartier divisor;
\item  $Gr(\Zed) = (Gr^{\sigma}(\Zed))_{\sigma \in \Sigma}$ is a tuple indexed by elements of $\Sigma$, where for each embedding $\sigma \in \Sigma$, we have a \emph{Green function} $Gr^{\sigma}(\Zed)$ on $\mathcal X(\C^{\sigma})$.  By definition, this means $Gr^{\sigma}(\Zed)$ is a $(0,0)$-current on $\mathcal X(\C^{\sigma})$  which is  smooth  on the complement of $ \Zed(\C^{\sigma})$, has logarithmic singularities along $\Zed(\C^{\sigma})$, and satisfies Green's equation (as  currents\footnote{ More precisely, we suppose that we are given a `uniformization' $\mathcal X(\C^{\sigma}) = [ \Gamma \backslash X]$ presenting $\mathcal X(\C^{\sigma})$ as an orbifold; here $X$ is a Riemann surface and $\Gamma \subset \Aut(X)$ is a discrete group which contains a finite-index subgroup that acts without fixed points. We then interpret \eqref{GrGenEqn} to mean an equality of $\Gamma$-invariant currents on $X$, cf.\ \cite[\S 2.3]{KRYbook}.  }
 on $\mathcal X(\C^{\sigma})$):
\begin{equation} \label{GrGenEqn}
 d d^c \ Gr^{\sigma }(\Zed) \ + \ \delta_{\Zed(\C^{\sigma})} \ = \ \left[ \omega^{\sigma} \left({\Zed} \right) \right]
 \end{equation}
for some \emph{smooth} $(1,1)$-form $\omega^{\sigma}(\Zed)$;
\item let $\iota\colon \C \to \C$ denote complex conjugation, which induces a map $\iota \colon \mathcal X(\C^{\iota \circ \sigma}) \ \to \ \mathcal X(\C^{\sigma})$.
We then require that 
\[ Gr^{\iota \circ \sigma }(\Zed) \ =\ \iota^* Gr^{\sigma}(\Zed), \qquad \qquad \text{for all } \sigma \in \Sigma.  \]
\end{compactenum} 
\end{definition}

An important class of examples is the \emph{principal divisors}: if $f$ is a rational function on $\mathcal X$, then
\[ \widehat{div}(f)  \ = \ \left( div(f), \ ( \log |\sigma(f)|^2 )_{\sigma \in \Sigma} \right) \]
is an arithmetic divisor. 

\begin{definition} We define the first arithmetic Chow group $\CHR(\mathcal X)$, with \emph{real coefficients}, to be the quotient of the space of arithmetic divisors modulo the real subspace spanned by the principal divisors. 
\end{definition}

This space admits a perfect pairing, the so-called \emph{(Arakelov-Gillet-Soul\'e) intersection pairing}
\[ \la \cdot, \cdot \ra \colon \ \CHR(\mathcal X) \times \CHR(\mathcal X) \ \to \ \R, \]
whose construction we briefly recall, at least in the following special case: suppose  $\widehat Z_1 = (Z_1, Gr(Z_1))$ and $ \widehat Z_2 = (Z_2, Gr(Z_2)) $ are prime divisors equipped with Green functions, such that their supports on the generic fibre are disjoint. 
Then the intersection pairing is given by
\begin{align}
 \la \widehat Z_1, \widehat Z_2 \ra \ = \ \sum_{ \substack{\lie{p} \subset O_L  \\ \text{prime}}}  \  \sum_{x \in Z_1 \cap Z_2 (\F_{\lie{p}}^{alg})}  \log (N(\lie{p}))   \frac{ \length(\mathcal O_{ Z_1 \cap Z_2, x} ) } { \# \Aut(x) }  \notag \\ 
 + \frac12 \ \sum_{\sigma \in \Sigma} \ \int_{\mathcal X(\C^{\sigma})} Gr^{\sigma}(Z_1) \cdot \omega^{\sigma}(Z_2) + Gr^{\sigma}(Z_2) \delta_{Z_1(\C^{\sigma})}.
 \end{align} 
This definition can be extended to the space of arithmetic divisors by linearity, and with the aid of a `moving lemma' in this context, it can be shown that one obtains a well-defined pairing on $\CHR(\mathcal X)$, cf.\ \cite{GilletSoule}. 

The intersection pairing allows us to define a convenient decomposition of $\CHR(\mathcal X)$, following \cite[\S 4.1]{KRYbook}. For each $\sigma \in \Sigma$, we fix a volume form $\mu^{\sigma}$ on $\mathcal X(\C^{\sigma})$, and we fix a metrized line bundle 
\[ \widehat\omega \in \widehat{Pic}(\mathcal X) \otimes \R \]
such that for each $\sigma$, the corresponding first Chern class is $\mu^{\sigma}$. Abusing notation, we denote again by $\widehat\omega$ the image of $\widehat\omega$ under the natural map $\widehat{Pic}(\mathcal X)\otimes \R \to \CHR(\mathcal X)$. 

Denote by $\widehat{\mathbf 1} = (0, \mathbf 1) \in  \CHR(\mathcal X)$ the class corresponding to the constant function $1$ on each component $\mathcal X(\C^{\sigma})$. Let $Vert$ denote the real subspace of $\CHR(\mathcal X)$ spanned by $\widehat{\mathbf 1}$, together with classes of the form $(\mathcal Y, 0)$, where $\mathcal Y$ is an irreducible component of a reducible fibre $\mathcal X_{\lie{p}}$ at a finite prime $\lie{p}$. 

Next, we let $An$ denote the subspace spanned by classes of the form $\hat{\mathbf f} = (0, (f^{\sigma})_{\sigma \in \Sigma})$, where each $f^{\sigma}$ is a $C^{\infty}$ function on $\mathcal X(\C^{\sigma})$ such that 
\begin{compactenum}[(i)]
\item $ f^{\iota \circ \sigma} =  \iota^* \circ f $, and 
\item for each $\sigma \in \Sigma$, we have
\[ \int_{\mathcal X(\C^{\sigma})} \ f^{\sigma} \cdot \mu^{\sigma} = 0. \]
\end{compactenum}
If we set 
\[ \widetilde{MW} \ := \ \left( \R \widehat\omega \oplus Vert  \oplus An \right)^{\perp}, \]
we obtain an orthogonal decomposition into three subspaces:
\begin{equation} \label{ChowDecompEqn}
\CHR(\mathcal X) \ = \  \widetilde{MW} \ \oplus \ \left( \R \widehat\omega \oplus Vert \right) \  \oplus  \ An .
\end{equation}
For convenience, we refer the collection of elements $\{ \widehat \omega,\widehat{\mathbf 1} ,  (\mathcal Y,0), \hat{\mathbf f} \}$ as above as a \emph{good spanning set} for the latter two summands. For a class $\widehat \Zed \in \CHR(\mathcal X)$, we shall take the \emph{components of} $\widehat \Zed$ to mean the collection of pairings
\[ \la \widehat \Zed, \lie{Y} \ra \] 
where $\lie{Y}$ varies among the good spanning vectors, together with the orthogonal projection $ (\widehat\Zed)_{MW} \in \widetilde{MW} $. A class is then completely determined by its components.

Finally, we observe that the space $\widetilde{MW}$ admits a more geometric interpretation. Let 
\[ res_{L}  \colon \ \CHR(\mathcal X) \to CH^1(\mathcal X_L) \otimes_{\Z} \R , \qquad (Z, Gr(Z) ) \mapsto Z_{L} \]
denote the `restriction-to-the-generic-fibre' map. Note that if $\widehat{Z} = (Z, Gr(Z)) \in \widetilde{MW}$, then
\[ 0 = \la \widehat Z, \ \widehat{\mathbf 1} \ra  \ = \ \frac{[L:\Q]}{2}  \deg( Z_L)  \]
and so $res_{L}$ defines a map
\begin{equation} \label{resLJacEqn}
 res_{L} \colon \widetilde{MW} \to Jac(\mathcal X_L)(L) \otimes_{\Z} \R,
\end{equation}
 where $Jac(\mathcal X_L)(L)$ denotes the $L$-valued points of the Jacobian variety of $\mathcal X_L$. This latter map is an \emph{isomorphism}, which moreover identifies the intersection pairing with the negative of the N\'eron-Tate height pairing on $Jac(\mathcal X_L)(L) \otimes \R$, cf.\ \cite{Hriljac}.

\subsection{Shimura curves} {\ \\}

We begin this section by introducing (an arithmetic model of) the  Shimura curve attached to $B$, defined as a moduli space of abelian surfaces.
\begin{definition}  \label{shCurveDef}
Let $\CBZ$ denote the moduli space over $Spec(\Z)$ that attaches to a scheme $S $ the category whose objects are tuples 
\[ \CBZ(S) = \{ \underline A = (A, \iota_A) \}; \]
here $A$ is an abelian scheme over $S$ of relative dimension 2, and $ \iota_A \colon\OB \to End(A) $ is an $\OB$ action. We further require that for any $b \in \OB$, the characteristic polynomial for the action of $\iota(b)$ on $Lie(A)$ is given by
\[ ch( \iota(b)|_{Lie(A)}) = X^2 - Trd(b) X + Nrd(b) \in \mathcal O_S[X]. \]
Morphisms in this category are $\OB$-equivariant isomorphisms of abelian schemes over $S$.
\end{definition}

We record a few facts about $\CBZ$ here. 
\begin{theorem}[\cite{KRYbook}, Proposition 3.1.1] The moduli problem $\CBZ$ is representable by a Deligne-Mumford stack (also denoted by $\CBZ$), which is regular, proper and flat over $\Spec(\Z)$ of relative dimension 1, and smooth over $\Spec \Z[ D_B^{-1}] $.
\end{theorem}

The complex points $\CBZ(\C)$ can be described via \emph{complex uniformization}, as follows. Fix once and for all an isomorphism 
\begin{equation} \label{BRisomM2}
B_{\R} \isomto M_2(\R),
\end{equation}
 and for any $z = x+iy \in \lie{H}^{\pm}$, let 
\begin{equation}
J_z \ := \  \frac{1}{y} \begin{pmatrix} x & -x^2 - y^2 \\ 1 & -x \end{pmatrix},
\end{equation}
so that $J_z^2 = -1$. Then for any $z \in \lie{H}^{\pm}$, we may produce a point $\underline A_z = (A_z, \iota_z) \in \CBZ(\C)$ as follows. We take $A_z = B_{\R} / \OB$ as a real torus, with the action $\iota_z$ of $\OB$ given by left-multiplication; the complex structure on $A_z$ is given by \emph{right} multiplication by $J_z$. 

Conversely, a point $(A, \iota) \in \CBZ(\C)$ admits an $\OB$-linear isomorphism $A \simeq B_{\R} / \OB$ of real tori. A complex structure $h$ on $B_{\R}$ which commutes with the left-multiplication action of $\OB$ is necessarily given by right multiplication by some element $J \in B_{\R}$ with $J^2 = -1$. Under the isomorphism \eqref{BRisomM2}, there is a unique $z \in \lie{H}^{\pm}$ such that $J = J_z$.

By carefully tracking the choices made in the discussion above, we obtain the following fact, cf.\ \cite[ \S 3.2 ]{KRYbook}.

\begin{proposition}
The map $z \mapsto (A_z, \iota_z)$ induces an isomorphism
\begin{equation} \left[ \OBx \big\backslash \lie{H}^{\pm} \right] \ \isomto \ \CBZ(\C). 
\end{equation}
\qed
\end{proposition}

A certain metrized line bundle $\widehat\omega^o$, known as the  \emph{Hodge bundle}, will play a particularly important role in this paper, so we take the opportunity here to make our acquaintance.  Let $\mathscr A / \CB$ denote the universal abelian variety, with unit section $\epsilon \colon \CB \to \mathscr A$. We then set
\[ \omega^o \ :=  \ \epsilon^* \Omega^2_{\mathscr A / \CB} \in Pic(\CBZ). \]

We may equip the the Hodge bundle with the metric $|| \cdot ||$ described in \cite[\S 3.3]{KRYbook} (as the specific form of this metric plays almost no role in the present paper, we will not define it here), and obtain a class
\[ \widehat\omega^o = (\omega^o, ||\cdot||) \in \widehat{Pic}(\CBZ), \]
which we also view as an arithmetic divisor $\widehat\omega^o \in \CHR(\CBZ)$ via the natural map
\[ \widehat{Pic}(\CBZ)  \otimes \R \to \CHR(\CBZ). \]

\subsection{Orthogonal special cycles} {\ \\}

In this section, we recall the construction of  orthogonal special cycles and their Green functions.

\begin{definition}
Fix an integer $n>0$. Let $\Zed^o(n)^{\sharp}$ denote the DM stack over $Spec(\Z)$ representing the following moduli problem; to a scheme $S$, we associate the category of tuples
\[ \Zed^o(n)^{\sharp}(S) = \{ (\underline A, x) \} \] 
where $\underline A = (A, \iota) \in \CB(S)$ and $x \in End_{\OB}(A)$ is traceless, commutes with the $\OB$-action, and satisfies $x^2 + n=0$. 
\end{definition}

We may view $\Zed^o(n)^{\sharp}$ as a cycle on $\CB$ via the forgetful map. 

\begin{remark} \label{CohenMacauleyRmk}
As shown in \cite{KRpadic}, these cycles may contain 0-dimensional embedded components in fibres at primes $p|D_B$. To rectify this defect, we replace $\Zed^o(n)^{\sharp}$ by its Cohen-Macauleyfication $\Zed^o(n)$, so that $\Zed^o(n) \to \CB$ is a relative divisor. From the point of view of intersection theory, the cycles $\Zed^o(n)^{\sharp}$ and $\Zed^o(n)$ are indistinguishable, cf.\ \S 4 of \emph{loc.\  cit.} $\diamond$
\end{remark}

These cycles admit \emph{complex uniformizations}:
\begin{proposition}[cf.\ \S 3.4 of \cite{KRYbook}] \label{orthCpxUnifProp}
Let
\[ \Omega^o(n) := \left\{ \xi \in \OB \ | \ Trd(\xi) = 0, \ Nrd(\xi) = n \right\}, \]
and note that $\OBx$ acts on this set by conjugation. 
There is an isomorphism 
\[  \left[ \OBx \backslash \coprod_{\xi \in \Omega^o(n) } D_{\xi} \right] \ \simeq \ \Zed^o(n)(\C), \]
where 
\[ D_{\xi} = \left\{ z \in \lie{H}^{\pm}  \ | \ J_z \xi = \xi J_z \right\} \subset \lie{H}^{\pm}. \]
This isomorphism is compatible with the complex uniformization $\CB(\C) \simeq [ \OBx \backslash \lie{H}^{\pm}]$.
\end{proposition}

Next, we describe the Green functions constructed in \cite[ \S 3.5]{KRYbook}. Let $B^0 = \{ b \in B \ | \ Trd(b) = 0 \}$, which can be equipped with the quadratic form of signature $(1,2)$ given by the restriction of the reduced norm on $B$ to this subspace. Let $\D^o(B^0_{\R})$ denote the space of negative-definite (real) planes in $B^0_{\R} = B^0 \otimes_{\Q} \R$. 
It is straightforward  to check that the map 
\[ \lie{H} \isomto \D^o(B^0_{\R}) , \qquad z \mapsto J_z^{\perp} \]
is an isomorphism. As $J_{\overline z} = - J_z$, we obtain a 2-to-1 map
\[ \lie{H}^{\pm} \to \D^o(B^0_{\R}) , \qquad z \mapsto J_z^{\perp}. \]
Now suppose $v \in B^0_{\R}$, and $\zeta = J_z^{\perp} \in \D^o(B^0_{\R})$. Let $v_{\zeta}$ denote the orthogonal projection of $v$ onto $\zeta$, and set
\begin{equation} \label{RoDefEqn}
 R^o(v, z) := - 2 Nrd( v_{\zeta}), \qquad \text{where } \zeta = J_z^{\perp}.
\end{equation}
By construction, $v \in \zeta^{\perp}$ if and only if $R^o(v, z) = 0$; in particular, for  $\xi \in B^0$, we have $z \in \D_{\xi}$ if and only if $R^o(\xi, z) = 0$.

 For $r \in \R_{>0}$, let
\[ \beta_1(r) = \int_1^{\infty} e^{-ur} \frac{du}{u}. \]
As $r$ approaches 0, we have $\beta_1(r) = - \log r  - \gamma + O(r)$, where $\gamma$ is the Euler-Mascheroni constant.

The Green function $Gr^o(n,v)$ for the cycle $\Zed^o(n)$, which depends on a real parameter $v \in \R_{>0}$, is then given by
\begin{equation} \label{GrODef}
 Gr^o(n,v)(z) := \sum_{\xi \in \Omega^o(n)} \ \beta_1 \left( 2 \pi v R^o(\xi, z) \right);
 \end{equation}
 here we are viewing $Gr^o(n,v)$ as a $\OBx$-invariant function on $\lie{H}$.
In particular, we have the following equality of currents on $\CBZ(\C) \simeq [\OBx \backslash \lie{h}]$:
\begin{equation} \label{oSCgrEqn}
 d d^c [Gr^o(n, v)]  + \delta_{\Zed^o(n)(\C)} = [ \psi^o(n, v) c_1(\widehat\omega^o)], 
\end{equation}
where
\[ \psi^o(n, v)(z) = \sum_{\xi \in \Omega^o(n)} \left( 4 \pi v \left\{ R^o(\xi, z) + 2 Nrd(\xi) \right\} - 1 \right) e^{- 2 \pi v R^o(\xi, z)} \]
and $c_1(\widehat\omega^o)$ is the first Chern form of the Hodge bundle $\widehat\omega^o$. 

\subsection{The orthogonal generating series of Kudla-Rapoport-Yang} {\ \\}

Following \cite[ \S 3.5]{KRYbook}, we define the terms in the orthogonal generating series as follows.
For $n \in \Z_{>0}$ and $v \in \R_{>0}$, let
\[ \widehat{\Zed^o}(n, v) \ := \  ( \Zed^o(t), Gr^o(n,v) ) \ \in  \ \CHR(\CBZ)_{\R}. \]
For $n < 0$, note that the right hand side of \eqref{GrODef} defines a smooth function on $\CBZ(\C)$. Therefore, in this case we  define
\[ \widehat{\Zed^o}(n, v) := (0, Gr^o(n, v)). \]
Finally, for the constant term, we set
\[ \widehat{\Zed^o}(0, v) := - \widehat{\omega}^o -(0, \log v)  - (0, \log D_B). \]
\begin{theorem}[Theorem A, \cite{KRYbook}]
	\label{KRYThmA}
	Let $\tau = u + iv \in \lie{H}$. Define the ``\emph{orthogonal generating series}"
	\[ \Phi^o(\tau)  := \sum_{n \in \Z} \widehat{\Zed^o}(n,v) \ q^n,\]
	where $q = e^{2 \pi i \tau}$. Then $\Phi^o(\tau)$ is the $q$-expansion of a modular form for the group $\Gamma_0(4D_B)$,  of weight 3/2 and trivial character.
\end{theorem}

As the coefficients of $\Phi^o(\tau)$ lie in the infinite-dimensional space $\CHR(\CBZ)$, we need to explain the meaning of this theorem more precisely. 
Recall that in \refSec{chowSec}, we described a decomposition
\[ \CHR(\CBZ)= (Vert \oplus \R \widehat\omega^o ) \ \oplus \ An \ \oplus \ \widetilde{MW}, \]
together with a spanning set for the first two factors. This allows us to decompose our generating series $\widehat\Phi{}^o$ into components (given by the generating series formed by pairing the coefficients against the vectors in this spanning set, together with the projections to the $\widetilde{MW}$-component). Each of these generating series has  scalar coefficients with the exception of the $MW$-component, whose coefficients are valued in a finite-dimensional vector space.  The content of this theorem is that each of these generating series is the $q$-expansion of a modular form in the usual sense, and theorem is proved by identifying the generating series with the $q$-expansions of known modular forms. 

Moreover, the proof of this theorem reveals that the components arising from elements of $Vert$ and the $\widetilde{MW}$-component are \emph{holomorphic}.

\section{Unitary cycles and their Green functions} 

\subsection{Unitary special cycles} \label{UcyclesSec} {\ \\ }
 
In this section, we recall the construction of the unitary cycles as in \cite[\S2]{San1}; they are the analogues of \emph{unitary (Kudla-Rapoport) special cycles}, as described in \cite{KRUglob}, in the Shimura curve setting. 
Throughout this section, we fix an optimal embedding $\phi \colon o_k \to \OB$. We also fix an element $\theta \in \OB$ such that $\theta^2 = - D_B$. We may assume, without loss of generality, that
\[ Trd( \theta \ \phi(\sqrt{\Delta}) ) > 0\]
since (i) the assumption that $B$ is indefinite implies that the above trace is non-zero, and (ii) we may replace $\theta$ by $-\theta$ if necessary.

Suppose $\underline A = (A, \iota) \in \CB(S)$, for some base scheme $S$ over $\Spec(o_k)$. Then, by \cite[\S 3.1]{How1}, there exists a unique principal polarization $\lambda_A^0$ on $A$ such that 
\begin{equation} \label{lambda0Eqn}
(\lambda^0_A)^{-1} \circ \iota(b)^{\vee} \circ \lambda^0_A = \iota(\theta^{-1}  \ b^{\iota} \ \theta  ) , \qquad \text{for all } b \in \OB.
\end{equation}
By \cite[Lemma 2.4]{San1}, the isogeny
\begin{equation} \label{newPolEqnDef}
 \lambda_{A, \phi} := \lambda_{A}^0 \ \circ \ \iota \left( \theta \ \phi(\sqrt{\Delta}) \right) ;
\end{equation}
defines another (non-principal) polarization on $A$; note that the Rosati involution $*$ defined by this latter polarization satisfies
\[ \iota \left( \phi (a) \right)^* \ = \ \iota( \phi(a')) , \qquad \text{for all } a \in o_k .\]

Next, let  $\mathscr{E}^+$ denote the moduli stack over $\Spec(o_k)$ whose $S$-points parametrize triples 
\[ \mathscr{E}^+(S) = \Big\{ \underline E = (E, i_E, \lambda_E) \Big\};\] here $E/S$ is an elliptic curve, $i_E \colon o_k \to \End_S(E)$ is an $o_k$-action, and $\lambda_E$ is a principal polarization. We also require that
\begin{compactenum}[(i)]
\item the induced action $i_E \colon o_k \to \End_S(Lie(E))$ agrees with the action given via the structural morphism $o_k \to \mathcal O_S$, and
\item for any $a \in o_k$, we have $\lambda_E^{-1} \circ i_E(a)^{\vee} \circ \lambda_E = i_E(a')$.
\end{compactenum}
Similarly, we set $\mathscr{E}^-$ to be the moduli space of tuples $(E, i_E, \lambda_E)$ as above, except that we insist the $o_k$ act on $Lie(E)$ via the \emph{conjugate} of the structural morphism. 

Finally, we take
\[ \mathscr E \ := \ \mathscr E^+ \ \coprod \ \mathscr E^- \]
to be the disjoint union of these two stacks. 

For a base scheme $S$ over $\Spec(o_k)$, suppose we are given two points $\underline E \in \mathscr{E}(S)$ and $\underline A \in \CB(S)$. We then have the space of \emph{special homomorphisms}
\[ \Hom_{\phi}(\underline E, \underline A) \ := \ \left\{ y \in \Hom(E, A) \ | \ y \circ i_E(a) \ = \ \iota_A( \phi (a) ) \circ y, \ \ \text{for all } a \in o_k \right\}. \]
This space comes equipped with an $o_k$-hermitian form $h_{E,A}^{\phi}$ defined by 
\[ h_{E,A}^{\phi} (\lie{s}, \lie{t}) \ := \ (\lambda_E)^{-1} \circ {\lie{t}}^{\vee} \circ \lambda_{A, \phi} \circ \lie s \ \in \End(E, i_E) \simeq o_k. \]
Let $q^{\phi}_{E, A}(x) := h^{\phi}_{E,A}(x,x)$ denote the corresponding quadratic form. 

\begin{definition}[Unitary special cycles] 
For an integer $m \in \Z_{>0}$, let $\Zed(m, \phi)$ denote the DM stack over $\Spec(o_k)$, whose $S$-points parametrize tuples
\[ \Zed(m, \phi)(S)  \ = \ \left\{ \ (\underline E, \  \underline A,  \ y  ) \ \right\} \]
where $ \underline E \in \mathscr E(S)$, $\underline A \in \CB(S)$, and $y \in \Hom_{\phi}(\underline E, \underline A)$ is a special homomorphism such that $q_{E,A}^{\phi}(y) = m$. 
\end{definition}

By \cite[Proposition 2.6]{San1}, the forgetful morphism $\Zed(m, \phi) \to \CBok$ allows us to view $\Zed(m, \phi)$ as a \emph{divisor} on $\CBok := \CBZ \times \Spec(o_k)$; abusing notation, we shall denote the divisor by the same symbol, and hope that the reader will be able to discern from the context whether a given instance of the notation refers to the stack or the divisor.

\subsection{Complex uniformizations and Green functions}  \label{UGrnFnSec} { \ \\ \\ }
 Fix an embedding $\phi \colon o_k \to \OB$, which induces an embedding $\phi \colon k \to B$. We start by describing an alternative complex uniformization of $\CB$. Let $B^{\phi} = (B, (\cdot, \cdot)_{\phi})$ denote the hermitian space whose underlying vector space is $B$ endowed with
the $k$-vector space structure given by the action
\[ a \cdot v := \phi(a) v, \qquad \forall a \in k, v \in B. \]
The hermitian form $(\cdot, \cdot)_{\phi}$ is given by the formula
\begin{equation} \label{hermFormEqn}
 (x, y)_{\phi} := \frac{\Delta}{2} Trd(x \ y^{\iota}) \ + \ \frac{\sqrt{\Delta}}{2} Trd( \phi(\sqrt{\Delta}) \ x \ y^{\iota}) \in k.  
\end{equation}
Let $Q(x) =  (x,x)_{\phi} = \Delta Nrd(x)$ denote the corresponding quadratic form, and note that the $k_{\R}$-hermitian space $B^{\phi}_{\R} := B^{\phi} \otimes_{\Q} \R \simeq M_2(\R)$ is of signature $(1,1)$ with respect to this form. 

Let $\D(B^{\phi}_{\R})$ denote the space of non-isotropic $k_{\R}$-lines in $B^{\phi}_{\R}$. There is a transitive action of $B_{\R}^{\times}$ on $\D(B^{\phi}_{\R})$ 
given by \emph{right} multiplication:
\[ \gamma \cdot \zeta := \zeta \gamma^{-1}\subset B_{\R} \]
Recall that we had fixed an embedding $\sigma_0\colon k \to \C$; let $\mathbb I_k := \sqrt{\Delta} \otimes |\Delta|^{-1/2} \in k_{\R}$ denote the square root of $-1$ in $k_{\R}$ such that $\sigma_0(\mathbb I_k) = i$.
It is then straightforward to check that we have a $B_{\R}^{\times}$-equivariant isomorphism
\begin{equation} \label{altCpxUnifMapEqn}
\lie{H}^{\pm} \isomto \D(B^{\phi}_{\R}) \qquad z \mapsto  \zeta:= \left\{ v \in B_{\R} \ | \ v J_z = - \phi(\mathbb I_k) v \right\}.
\end{equation}
Without loss of generality, we may normalize the map \eqref{altCpxUnifMapEqn} so that $i$ is identified with the line $\phi(k) \otimes \R \subset B_{\R}$. With this choice, $\lie{H}^+$ is identified with the subset $\D(B^{\phi}_{\R})^-\subset \D(B^{\phi}_{\R})$ consisting of negative-definite lines.  
In particular, we obtain an alternative uniformization 
\[ \CBZ(\C) \simeq [ \OBxone \big\backslash \D(B^{\phi}_{\R})^- ],\]
where $\OBxone$ is the subgroup of $\OBx$ consisting of elements of norm 1, acting on $\D(B_{\R}^{\phi})^-$ by right multiplication.

The geometric meaning of this isomorphism is as follows. Any complex abelian surface $A \in \CBZ(\C)$ admits an isomorphism $A \simeq B_{\R} / \OB$ of real tori. The embedding $\phi$ then determines a \emph{distinguished} complex structure on $B_{\R}$, namely right-multiplication by $-\phi(\mathbb I_k)$, which commutes with the action of $\OB$. Any other complex structure, in particular the one arising from the complex structure of $A$, is then determined by the line in $B_{\R}$ on which it agrees with the distinguished one. Hence, after possibly applying an automorphism to ensure that this line is negative, we see that the space of such lines in turn parametrizes the points of $\CBZ(\C)$. 

Upon fixing coordinates, an easy calculation  yields the following description of the first Chern class of $\widehat\omega^o$ under this uniformization, cf.\ \cite[(3.16)]{KRYcomp}.
\begin{lemma} \label{hodgeUnifLemma}
Fix a basis $\{ e, f \}$ of $B^{\phi}_{\R}$ such that $(e,e)_{\phi}  = -(f,f)_{\phi} = 1$ and $(e,f)_{\phi} = 0$. We then obtain an isomorphism
\[ U^1 := \{ \lie{Z} \in \C \ \mid \ |\lie{Z}| < 1 \} \ \isomto \D(B^{\phi}_{\R})^-, \qquad \lie{Z} \mapsto {\rm span}_{k_{\R}}(\lie{Z}e + f). \]
With respect to this coordinate, we have
\begin{equation} \label{chern1Eqn}
 c_1(\widehat \omega^o) = \frac{i}{\pi} \frac{d\lie{Z} \wedge d \overline{\lie{Z}}}{(1-|\lie{Z}|^2)^2}.
\end{equation}
 \qed
\end{lemma}
We now say a few words about polarizations. Let $\underline A_{\zeta} \in \CBZ(\C)$ denote the abelian surface corresponding to $\zeta \in \D(B_{\R}^{\phi})^-$, and  recall that fixing an element $\theta \in \OB$ with $\theta^2 = -D_B$ determines a principal polarization $\lambda^0 = \lambda^0_{A_{\zeta}}$ on $A_{\zeta}$  as in \eqref{lambda0Eqn}. Then the Riemann form $\{ \cdot , \cdot \}_{\lambda^0}$ on $H_1(A_{\zeta}) = \OB$ corresponding to $\lambda^0$ is given by
\[ \{ a, b \}_{\lambda^0} \ = \ Trd \left( \ a \ b^{\iota} \ \theta^{-1} \right). \]
Thus, if we set 
$
 \lambda_{A, \phi} := \lambda^0 \ \circ \ \iota \left( \theta \ \phi(\sqrt{\Delta}) \right),
$
then we find that the corresponding Riemann form is 
\begin{equation} \label{RiemFormEqn}
 \{ a, b \}_{\lambda_{A,\phi}} \ = \ Trd \left( \phi(\sqrt{\Delta}) \   a \ b^{\iota}  \right) \ = \ tr_{o_k/\Z}  \left( (\phi(\sqrt{\Delta}) a, \ b )_{\phi} \right) .
\end{equation}

We now turn to the complex uniformization of the cycles $\Zed(m,\phi)$. Recalling that $\mathcal E = \mathcal E^+ \coprod \mathcal E^-$, we have a corresponding decomposition
\[ \Zed(m, \phi) \ = \ \Zed^+(m, \phi)  \ \coprod \ \Zed^-(m, \phi) \]
of stacks over $\Spec(o_k)$. We also recall our convention that $\Zed^{\pm}(m, \phi)(\C^{\sigma_0})$ denotes the complex points of $\Zed^{\pm}(m, \phi)$, where we view $\C$ as an $o_k$-algebra via the fixed embedding $\sigma_0 \colon k \to \C$.

\begin{proposition} \label{cpxUnifProp}
For a fractional ideal $\lie{a}$ of $o_k$, and an integer $m>0$, let 
\[ \Omega^+(m, \lie{a}, \phi):= \left\{ y \in \phi(\lie{a})^{-1} \OB \ \Big{|}  \ Nrd(y)  = \frac{m}{\Delta N(\lie{a})}  \right\}, \]
which admits an action of $o_k^{\times} \times \OBxone$  via the formula
\[ (a,b) \cdot y := \phi(a)  \ y \ b^{-1}, \qquad \qquad  a \in o_k^{\times},  \ b \in \OBxone. \]
Then
\[ \Zed^+(m, \phi)(\C^{\sigma_0}) \simeq \left[ o_k^{\times} \times \OBxone \Big\backslash \coprod_{[\lie{a}] \in Cl(k)} \coprod_{y \in \Omega^+(m, \lie{a}, \phi) } \D_{y} \right], \]
where $\D_{y} := \left\{ (y)^{\perp}  \right\} \subset \D(B^{\phi}_{\R})^-, $
and the ideals $\lie{a}$ range over any set of representatives for the class group.
Similarly, if 
\[ \Omega^-(m, \lie{a}, \phi):= \left\{ y \in \phi(\lie{a})^{-1} \OB \ \Big{|}  \ Nrd(y)  = \frac{-m}{\Delta N(\lie{a})}  \right\}, \]
then
\[ \Zed^-(m, \phi)(\C^{\sigma_0}) \simeq \left[ o_k^{\times} \times \OBxone \big\backslash \coprod_{[\lie{a}] \in Cl(k)} \coprod_{y \in \Omega^-(m, \lie{a}, \phi) } \D_{y}' \right] ,\]
where $\D_{y}' := \left\{ \text{span}_{k_{\R}} ( y) \right\} $.

\begin{proof}
We begin by describing the uniformization of $\Zed^+(m, \phi)$. 
Fix a set of representatives $\{ \lie{a}_1, \dots, \lie{a}_h \}$ for the class group $Cl(k)$ of $k$. For each $\lie{a}_i$, we have a complex elliptic curve 
\[ E_i = \C/ \sigma_0(\lie{a}_i). \]
There is also a natural $o_k$-action of signature $(1,0)$, given by the restriction of the embedding $\sigma_0 \colon k \to \C$ to $o_k$, and a compatible principal polarization $\lambda_{E_i}$ determined by the alternating Riemann form
\begin{equation} \label{EpolRFEqn}
 \{x,y\}_{E_i} \mapsto \frac{1}{2N(\lie{a}_i) \Delta} tr_{\C / \R}( x \ y' \ \sigma_0(\sqrt{\Delta})), \qquad x, y \in H_1(E_i, \Z) = \lie{a}_i.
\end{equation}
In this way, the triple $\underline E_i := (E_i, \sigma_0|_{o_k}, \lambda_{E_i})$ defines a point of $\mathscr E^+(\C^{\sigma_0})$, with automorphism group $|o_k^{\times}|$, and every point of $\mathscr E^+(\C^{\sigma_0})$ is isomorphic to such a triple. 

Now suppose $(\underline E, \underline A, \lie{y}) \in \Zed^+(m, \phi)(\C^{\sigma_0})$ is a complex-valued point of a unitary special cycle. We may assume that $\underline A = \underline A_{\zeta}$ corresponds to a line $\zeta \in \D(B^{\phi}_{\R})^-$,  and $\underline E = \underline E_i$ as above. Then $\lie y \colon \underline E_i \to A_{\zeta}$ is determined by the $o_k$-linear map
\[ \tilde y \colon \lie{a}_i \simeq H_1(E_i, \Z) \to H_1(A_{\zeta}, \Z) \simeq \OB \]
on the level of homology, which in turn is determined by the image of any non-zero element. For concreteness, let
\begin{equation} \label{lieCondEqn}
 y := \frac{1}{N(\lie{a}_i)} \ \tilde y(N(\lie{a}_i)) \in \phi(\lie{a}_i)^{-1} \OB. 
\end{equation}
If $J  \in End(B_{\R})$ denotes the complex structure on $Lie(A_{\zeta}) = B_{\R}$ determined by $\zeta$, then 
\begin{align*}
 J ( \xi) \ &=  \ \tilde y ( \sigma_0 (\mathbb I_k))  &  \text{[by holomorphicity]} &  \\
 &=  \ \phi(\mathbb I_k) y  & \text{[by } o_k \text{-linearity]} ; &
\end{align*}
in turn, these relations hold if and only if $\zeta  = ( y)^{\perp}$. 

Furthermore, by the description of the polarizations $\lambda_{E_i}$ and $\lambda_{A_z, \phi}$ as in \eqref{EpolRFEqn} and \eqref{RiemFormEqn}, we find
\begin{align*}
 h_{E_i, A_{\zeta}}^{\phi}( \lie y, \lie y) \ &= \  (\lambda_{E_i})^{-1} \circ  \lie y^{\vee} \circ \lambda_{A_z, \phi} \circ \lie y \\
\ &= \Delta N(\lie{a}_i)  Nrd(y),
\end{align*}
and so 
\[  h_{E_i, A_{\zeta}}^{\phi}( \lie y, \lie y) = m \implies y \in \Omega^+(m, \lie{a}, \phi). \]

To summarize, we have described a construction 
\[ \begin{pmatrix} \lie y \in Hom_{o_k}^{\phi}(E_i, A_{\zeta}) \\ {\rm with } \  h^{\phi}_{E_i, A_{\zeta}}( \lie y, \lie y) = m \end{pmatrix} \ 
\rightsquigarrow 
\begin{pmatrix} y \in \Omega^+(m, \lie{a}_i, \phi) \\ {\rm with } \  \zeta = y^{\perp} \end{pmatrix}.
\]

Conversely, an element $y \in \Omega^+(m, \lie{a}_i, \phi)$ defines a morphism $\lie{y} \colon E_i \to A_{\zeta}$ for $\zeta = y^{\perp}$, by demanding that the corresponding map on Lie algebras is determined by $y$ as in \eqref{lieCondEqn}. Finally, one sees that the action of $o_k^{\times} \times \OBxone$ on $\mathcal E^+(\C) \times \CB(\C)$
induces the action on $\Omega^+(m, \lie{a}_i, \phi)$ described in the proposition. 
The corresponding statement for $\Zed^-(m, \phi)$ follows by a similar argument.
\end{proof}
\end{proposition}
%

These uniformizations will allow us to furnish the cycles $\Zed^{\pm}(m,\phi)$ with Green functions, inspired by those appearing in \cite[\S 3.5]{KRYbook}. 
Given  $b \in B$ and a line $\zeta \in \D(B^{\phi}_{\R})^-$, we let $pr_{\zeta}(b)$ denote the orthogonal projection of $b$ onto $\zeta$, and  set
\[R_{\phi}(b, \zeta) = - 2 \left( pr_{\zeta}(b), \  pr_{\zeta}(b) \right)_{\phi} \geq 0.   \]
Note that $R_{\phi}(b, \zeta) = 0$ if and only if  $\zeta = b^{\perp}$. In addition, for any $\gamma \in GU(B^{\phi}_{\R})$, we have
\[ R_{\phi}(b, \gamma(\zeta)  ) = R_{\phi}(\gamma^* (b), \zeta). \]
Consider the exponential integral
\[ \beta_1(u) := \int_{1}^{\infty} e^{-ut} t^{-1} \mathop{dt}, \qquad u \in \R_{>0}, \]
and set
\[ Gr^+_{\phi}(b, \zeta) := \beta_1( 2 \pi R_{\phi}(b, \zeta)). \]

\begin{proposition} \label{ddcUnitary} Suppose $(b, b)_{\phi}>0$. As currents on $\D(B^{\phi}_{\R})^-$, we have
\[ d d^c [ Gr_{\phi}^+(b, \zeta) ]   + \delta_{\D_{b}} = [  \varphi_{\phi}^+ \cdot  c_1(\widehat{\omega}^o)], \]
where 
\[ \varphi_{\phi}^+( b, \zeta) := \frac{1}{2} \big( 2 \pi ( R_{\phi}(b, \zeta) + 2 \Delta Nrd(b)) -1 \big) \exp \big( - 2 \pi R_{\phi}(b, \zeta) \big), \]
 $c_1(\widehat \omega^o)$ is the first Chern form of the Hodge class, and $ \D_{b} = \left\{ b^{\perp} \right\}$ is a singleton set in $\D( B^{\phi}_{\R})^-$.  
\begin{proof} 
We fix a basis $\{ e, f \}$ of $B^{\phi}_{\R}$ as a $k_{\R}$-vector space such that $(e,e)_{\phi} = - (f,f)_{\phi} = 1$, and $(e,f)_{\phi} = 0$, which yields an identification $ U^1 \simeq \D(B^{\phi}_{\R})^- $ as in \refLemma{hodgeUnifLemma}.

Write $b = \xi_1 e + \xi_2 f$, with $\xi_1, \xi_2 \in k_{\R}$. Then if $\lie{Z} = \lie{Z}(\zeta) \in U^1$ corresponds to $\zeta$, we have
\[ R_{\phi}(b, \lie{Z}) := R_{\phi}(b, \zeta) = - 2(pr_{\zeta}(b) , \  pr_{\zeta}(b))_{\phi} =2 \frac{|\lie{Z} \xi_1 - \xi_2|^2}{1 - |\lie{Z}|^2}. \]
Thus away from $\D_{\xi}$, we have
\begin{align*} d d^c \ Gr_{\phi}^+(b, \lie{Z}) \ =&  \ \frac{i}{ 2 \pi}  \ \partial \bar{\partial}  \ Gr_{\phi}^+(b, \lie{Z})  \\
 =& \  \ i \ \frac{e^{-2 \pi R}}{2 \pi R} \left[ \ - \frac{\partial^2 R}{\partial \lie{Z} \partial \bar{\lie{Z}} } + \left( 2 \pi + \frac{1}{R} \right) \frac{\partial R}{\partial \lie{Z}} \frac{\partial R }{\partial \bar{\lie{Z}}}  \right] \ d\lie{Z} \wedge d \bar{\lie{Z}} ,
  \end{align*}
where $R = R_{\phi}(b, \lie{Z})$.
Evaluating the derivatives in the preceding display, we see that away from $\D_{\xi}$,
\begin{equation} \label{ddcGrPhiEqn}
 d d^c Gr_{\phi}^+(b, \lie{Z}) = i \left( R_{\phi}(b, \lie{Z}) + 2\Delta Nrd(b) - \frac{1}{2 \pi} \right) e^{-2 \pi R} \frac{ d\lie{Z}\wedge d \bar{\lie{Z}}}{ (1- |\lie{Z}|^2)^2}. 
\end{equation}
Now to investigate the behaviour of $Gr_{\phi}^+(\xi, \lie{Z})$ near $\D_{\xi}$, we use the expression
\[ \beta_1(t) = - \gamma - \log t - \int_0^{-t} \frac{e^u-1}{u} du, \]
where $\gamma$ is the Euler-Mascheroni constant. Also note that the point in $w_0 \in U^1$ corresponding to the line $ \xi^{\perp} \in \D_{\xi}$ is given by $\lie{Z}_0 = \xi_2 / \xi_1$, and hence
\[ |\lie{Z} \xi_1 - \xi_2 |^2 = |\xi_1|^2 \ |\lie{Z} - \lie{Z}_0|^2. \]
In particular, we may write
\[ Gr_{\phi}^+(b,\lie{Z}) = -\log |\lie{Z} - \lie{Z}_0|^2 + f(\lie{Z}) \] where $f(\lie{Z})$ is a $C^{\infty}$ function. 

Let $N_{\epsilon}$ denote a small ball of radius $\epsilon$ around $w_0$, and $g$ a compactly supported $C^{\infty}$ function on $U^1$. Then, abbreviating $Gr^+ = Gr_{\phi}^+(b, \lie{Z})$, we have
\begin{align*} 
\int_{U^1} Gr^+ \cdot d d^c g &= \lim_{\epsilon \to 0} \int_{U^1 - N_{\epsilon}} Gr^+ \cdot d d^c g \\
&= \lim_{\epsilon \to 0} \int_{U^1 - N_{\epsilon}} g \cdot  d d^c Gr^+ \ - \int_{\partial N_{\epsilon}} \left( Gr^+ \cdot d^c g - g \cdot d^c Gr^+ \right).
\end{align*}

Consider the limit $\epsilon \to 0$. Combining \eqref{ddcGrPhiEqn} and \eqref{chern1Eqn}, the first integral approaches $[ \psi^+ \cdot   c_1(\widehat{\omega})](g)$, while using the ``Integral Table'' \cite[p.23]{Lang}, the second integral approaches $- g(w_0)$. 
\end{proof}
\end{proposition}
 
Note that for $\zeta \in \D(B^{\phi}_{\R})^-$, and a \emph{negative} vector $b \in B^{\phi}_{\R}$, we have $b \in \zeta$ if and only if $R_{\phi}(b, \zeta) +  2\Delta Nrd(b) = 0$. With this in mind, we set 
\[ Gr^-_{\phi}(b, \zeta) \ := \ \beta_1 \Big ( 2 \pi \ ( R_{\phi}(b, \zeta) + 2\Delta Nrd(b)) \Big) .\]
Then, by a similar calculation to the previous proposition, we have
\[ dd^c [Gr_{\phi}^-(\xi)] + \delta_{\D'_{\xi}} = [\varphi^- \cdot c_1(\widehat{\omega})], \] 
where
\[ \varphi_{\phi}^-( b, \zeta) = \frac{1}{2} \left( 2 \pi \ R_{\phi}(b, \zeta) -1 \right) \exp \big( - 2 \pi \ ( R_{\phi}(b, \zeta) + 2 \Delta Nrd(b)) \big).  \]

We want to use these calculations to describe Green functions for the unitary special cycles, and thereby define classes in the first arithmetic Chow group of $\CB$. One technical hiccup, however, is that the cycles are defined over $o_k$ and not $\Z$. The following lemma assures us that the machinery of arithmetic Chow groups continues to operate after base change to $o_k$.  
\begin{lemma} The base change $ \CBok := \CBZ \times_{\Z} \Spec( o_k)$ is an arithmetic surface over $\Spec(o_k)$ in the sense of \refSec{chowSec}: it is regular of dimension $2$, and proper and flat over $\Spec(o_k)$. Moreover, $\CBok$ is smooth over $\Spec(o_k[D_B^{-1}])$. 
\begin{proof}
The only non-obvious point is regularity. Note that $\Spec(o_k)$ is smooth over $\Spec \Z [|\Delta|^{-1}]$, and recall that $\CBZ$ is smooth over $\Spec \Z[D_B^{-1}]$. Since $(D_B, |\Delta|) = 1$,  every point $x \in \Spec(\Z)$ has an open neighbourhood $U$ such that at least one of $\CBZ \times U$ or $\Spec(o_k) \times U$ is smooth over $U$. Since regularity is a local property and smooth base changes preserve regularity, it follows that $\CBok$ is regular.
\end{proof}

\end{lemma}

For every $m > 0$ and $\eta \in \R_{>0}$, we define an arithmetic class
\begin{equation} \label{unDivClassPosEqn}
 \widehat\Zed(m, \phi, \eta) \ := \ \big( \Zed(m, \phi), \ Gr^u(m, \phi, \eta) \big) \ \in \ \CHR(\CBok),
\end{equation}
where the Green functions 
\[ Gr^u(m, \phi, \eta)  = (Gr^u(m, \phi, \eta)^{\sigma_0}, Gr^u(m, \phi, \eta)^{\sigma_1})\] 
are given as follows: set
\begin{align}  Gr^u (m, \phi, \eta)^{\sigma_0}(\zeta) \ := \ \frac{1}{|o_k^{\times}|}  \sum_{[\lie{a}] \in Cl(k)} 
 \sum_{y \in \Omega^{+}(m, \lie{a}, \phi)} & Gr_{\phi}^{+} \left( (N(\lie{a})\eta)^{1/2} y, \zeta \right)  \notag  \\
 &+  \sum_{y \in \Omega^{-}(m, \lie{a}, \phi)}  Gr_{\phi}^{-} \left( (N(\lie{a})\eta)^{1/2} y, \zeta \right)  \label{grnFnEqnDef}
 \end{align}
and, if  $\iota \colon \CBok(\C^{\sigma_1}) \to \CBok(\C^{\sigma_0})$ is the map induced by complex conjugation, we also set
\[ Gr^u (m, \phi, \eta)^{\sigma_1} \ :=  \iota^* Gr^u (m, \phi, \eta)^{\sigma_0}.  \]

Note that when $m<0$, the same formula \eqref{grnFnEqnDef} defines a \emph{smooth function} on $\CB(\C^{\sigma_0})$, and so in this case we may also define classes
\begin{equation} \label{unDivClassNegEqn}  \widehat\Zed(m, \phi, \eta) \ := \ \big( 0, \ Gr^u(m, \phi, \eta) \big) \ \in \ \CHR(\CBok). 
\end{equation}

\begin{remark} The normalizing factor $(N(\lie{a}))^{1/2}$ appears in the arguments of the functions $Gr^{\pm}$ in \eqref{grnFnEqnDef} because we choose to work with vectors in $B$, rather than the Hermitian space $Hom(\lie{a},B)$.
\end{remark}

\section{The Shimura lift, and statement of the main theorem} 

\subsection{The Shimura lift via theta series} \label{shLiftSec}
 {\ \\}
 
In this section, we rapidly recall the construction of the Shimura lift \cite{Shim}, which is an operation taking modular forms of half-integral weight to modular forms of even integral weight. Rather than using Shimura's original approach, however, we shall consider the reformulation due to Shintani \cite{Shin} and Niwa \cite{Niwa}, who show that one can recover the Shimura lift via integration against a specific theta kernel; this approach has the advantage of being readily generalized to forms that are non-holomorphic. 

We begin by recalling the construction of the Niwa-Shintani theta kernel. Fix an integer $N>0$, a squarefree positive integer $t$, and consider the quadratic form on $\R^3$ determined by the matrix
\[ Q = \frac{2}{Nt} \begin{pmatrix} & & -2 \\ & 1 & \\ -2 & & \end{pmatrix}. \] 
Let $L = \Z \oplus Nt \Z \oplus (Nt/4)\Z$, which is an even self-dual lattice for the form $Q$.

There is an action of $SL_2(\R)$ on $\R^3$, given as follows: for $g \in SL_2(\R)$, put $g \cdot (x_1, x_2, x_3) = (x_1', x_2', x_3')$, where
\[ \begin{pmatrix} x_1' & x_2'/2 \\ x_2'/2 & x_3' \end{pmatrix} = g \begin{pmatrix} x_1 & x_2/2 \\ x_2/2 & x_3 \end{pmatrix} g^{t}. \]
This action induces an isomorphism $SL_2(\R) / \left\{ \pm Id \right\} \simeq SO(Q)$. 

Finally, for any positive integer $\lambda$, define the Schwarz function $f_{\lambda} \in \mathcal S(\R^3)$ to be
\[ f_{\lambda}(x_1, x_2, x_3) := (x_1 - ix_2 - x_3)^{\lambda} \exp \left( - \frac{2 \pi}{Nt} ( 2x_1^2 + x_2^2 + 2x_3^2) \right). \]

\begin{definition}[Niwa-Shintani theta kernel]
Suppose $N$, $t$, and $\lambda$ are positive integers, and $t$ is squarefree. Let $\chi$ be a Dirichlet character modulo $N$, and set $\kappa = 2 \lambda + 1$. For $\tau = u+iv, w = \xi + i \eta \in \lie{H}$, define the \emph{Niwa-Shintani theta kernel} of parameter $t$ to be
\[ \Theta_{NS}(\tau, w) \ := \ (4\eta)^{-\lambda} \  v^{-\kappa / 4} \ \sum_{x \in L} \  \overline{ \chi_t}(x_1) \ \left\{ \omega( \sigma_{\tau}) f_{\lambda} \right\} (\sigma_{4w}^{-1} \cdot x),\]
where 
\begin{itemize}
\item $ \chi_t(a) := \chi(a) (-1/ a)^{\lambda} (t / a)$, where $({\cdot}/{\cdot})$ is the (modified) Kronecker symbol, cf.\ Appendix A of \cite{Cipra};
\item $\omega$ is the Weil representation for the dual pair $(SL_2, O(Q))$;
\item $\sigma_{\tau} = \begin{pmatrix} v^{1/2} & u v^{-1/2} \\ & v^{-1/2} \end{pmatrix}$;
\item and $\sigma_{4w} = \begin{pmatrix} 2 \eta^{1/2} & 2 \xi \eta^{-1/2} \\ & \eta^{-1/2} / 2 \end{pmatrix}$. 
\end{itemize}
\end{definition}

It follows from the properties of the Weil representation that as a function of $\tau$, we have $\Theta_{NS}$ is a modular form of weight $\kappa/2$, for the congruence group $\Gamma_0(4Nt)$, with character $\overline {\chi} (Nt / \cdot)$. 

As a function of $w$, the conjugate of the theta kernel $\overline{\Theta_{NS} }$ is a modular form of weight $2 \lambda = \kappa -1$, with level $2Nt$ and character $\overline{\chi}^2$.

To recover Shimura's lift, it turns out that one needs to apply a Fricke involution (normalized to take into account the half-integral weight) to both variables; following the notation of \cite{Cipra}, define
\begin{align*}  \Theta_{NS}^{\#}(\tau, w) &:= \Theta_{NS}|_{W(4Nt)} \overline{|_{W(2Nt)}}  \\  &=  2^{-2 \lambda -1/2} (Nt)^{-3 \lambda/2 - 1/4} (-i \tau)^{- \kappa/2} (\overline{w})^{-2 \lambda} \ \Theta_{NS} \left( \frac{-1}{4Nt\tau},\frac{-1}{2Ntw} \right). \end{align*}

\begin{definition}[Shimura lift] 
Suppose that $G \in M_{\kappa / 2}(4N,\chi)$, and let $G_t(\tau) := G(t \tau) \in M_{\kappa / 2}(4Nt, \chi(t/\cdot))$. We define the \emph{Shimura} lift of $G$ to be the function
\begin{equation} \label{shLiftFormula} Sh_t(G)(w) := C(\lambda)  \int_{ \Gamma_0(4Nt) \backslash \lie{H}} v^{\kappa / 2} \ G_t(\tau) \ \overline{ \Theta_{NS}^{\#}( \tau, w) } \ d \mu(\tau), \end{equation}
whenever the integral is defined; here $\lambda = (\kappa - 1)/2$ and  $C(\lambda) = (-1)^{\lambda} 2^{3 \lambda - 2} (tN)^{-\lambda /2 - 1/4}$. 
\end{definition}
%
%
%
%

We will also  need  a Poincar\'{e} series expression for the Niwa-Shintani theta kernel. The proof amounts to making the necessary straightforward modifications to the argument  in \cite[\S 3]{Koj}, where an analogous statement for the `untwisted' theta kernel appears; we omit the argument here, and refer the reader to  \cite[Theorem 4.4]{SanThesis} for details. 

\begin{theorem} \label{thetaPoincare} 
For any integer $\mu \geq 0$, $\tau = u+iv \in \lie{H}$, and $\alpha \in \R$, define a theta function
\[ \theta_{\mu}(\tau, \alpha) := (2 \sqrt{2 \pi})^{- \mu} \ v^{-\mu/2} \sum_{\ell \in \Z} H_{\mu} \left(2 \sqrt{2 \pi v}\  \ell \right) e \left( \tau \ell ^2 + 2 \alpha \ell \right) , \]
where $e(z) = e^{2 \pi i z}$ and $H_{\mu}$ is the Hermite polynomial 
\[ H_{\mu}(x) = (-1)^{\mu} \ e^{x^2/2} \ (d/ dx)^{\mu} \left(e^{-x^2/2} \right).\]
%
Then, for $w = \xi + i \eta \in \lie{H}$, 
\begin{align*} \Theta_{NS}^{\#}(\tau, w) =& \ C \ \sum_{\mu = 0}^{\lambda} \ \binom{\lambda}{\mu} 4^{\mu} \ \eta^{1-\mu} \sum_{m \in \Z} \overline{\chi_t(m)} \ m^{\lambda - \mu} \\
\times& \sum_{\gamma \in \Gamma_{\infty} \backslash \Gamma_0(4Nt)} \frac{\Im(\gamma \tau)^{\mu - \lambda} }{\chi_t(\gamma)j_{\kappa/2}(\gamma, \tau)} \exp \left( - \frac{\pi \eta^2 m^2}{4\Im (\gamma \tau)} \right) \theta_{\mu}( \gamma \tau, -\xi m /2),
\end{align*}
where $C =(-1)^{\lambda} 2^{-4 \lambda} (tN)^{\lambda/2 +1/4}$, and $\chi_t(\gamma) = \chi_t(d)$ when $\gamma = ( \begin{smallmatrix} a & b \\ c & d \end{smallmatrix}) \in \Gamma_0(4Nt)$. 

\qed
\end{theorem}

\subsection{The unitary generating series and the main theorem} {\ \\}

Recall that for $m \in \Z$ with $m \neq 0$, and any $\eta \in \R_{>0}$ we have constructed classes
\[ \widehat\Zed(m, \phi, \eta) = (\Zed(m, \phi), \ Gr^u(m, \phi, \eta) ) \ \in \ \CHR(\CBok), \]
as in \eqref{unDivClassPosEqn} and \eqref{unDivClassNegEqn} -- by convention, we set $\Zed(m, \phi) = 0$ when $m < 0$. 

We also require the following rescaled variants of these cycles: let
\begin{equation} \label{Z*DefEqn}
 \widehat\Zed{}^*(m, \phi; \eta) = (\Zed^*(m, \phi) , \  Gr^{u*}(m, \phi, \eta)), 
\end{equation}
where
\[ \Zed^*(m, \phi):=   \Zed \left( \frac{ m}{(|m|,D_B)}, \ \phi\right), \ \   \text{and } Gr^{u*}(m, \phi, \eta):= Gr^u \left(  \frac{ m}{(|m|,D_B)}, \ \phi, \  (|m|, D_B)^{1/2} \cdot \eta\right). \]
Note that if $(|m|, D_B) = 1$, then $\widehat\Zed{}^*(m, \phi,\eta)$ and $\widehat\Zed{}(m, \phi, \eta)$ coincide.

It remains to define the ``constant term" $\widehat\Zed(0, \eta)$. It is worth pointing out that there is a unique value for the constant term which makes our \hyperref[mainThm]{main theorem} true, and this is what prompts our definition; unfortunately, I am unaware of an \emph{a priori}  reason for this term to have the precise value that it does.

To start, let $\chi_k$ denote the Dirichlet character associated to $k$,  so that for a prime $p$, we have
\[  \chi_k(p) \ = \  \begin{cases} 1, & \text{if } p \text{ is split,} \\ 0, & \text{if } p \text{ is ramified,} \\ -1, & \text{if } p \text{ is inert.} \end{cases}\]
Consider the character $\chi'$ obtained by induction to level $4D_B |\Delta|$; that is, 
\begin{equation} \label{chi'DefEqn}
 \chi'(a) = \begin{cases} 0 ,  & \text{if } (a, 4D_B|\Delta|) > 1 \\ \chi_k(a), & \text{if } (a, 4D_B|\Delta|) = 1 . \end{cases}
\end{equation}
We also define the Gauss sum
\[ \check{\chi}'(a) := \sum_{h = 0}^{4 D_B |\Delta| - 1}  \chi'(h)  \exp \left( \frac{ 2 \pi i a h }{ 4 D_B |\Delta|} \right)  \]
and its ``L-function" \footnote{Note $\check{\chi}'$ is not a Dirichlet character, so $L(s, \check{\chi}')$ is not an $L$-function in any meaningful sense.}
\[ L(s, \check{\chi}') := \sum_{m > 0} m^{-s} \check{\chi}'(m), \]
which is analytic (as written) in the half-plane $Re(s)>0$. 
%

\begin{definition}  \label{uConstDef}
Let $\widehat\omega := \widehat\omega^o_{/o_k}$ denote the base change of (the class corresponding to) the Hodge class $\widehat{\omega}^o$ to $o_k$, and define the constant term to be
\[ \widehat \Zed(0, \eta) \ := \  \frac{i}{2 \pi} \ L(1, \check{\chi}') \ [  \widehat\omega + (2 \log(\eta) + A)\cdot  \widehat{\mathbf 1}],\]

where the class $\widehat{\mathbf 1} = (0, \mathbf 1)$ corresponds to the constant function 1,  and 
\[ A = \log ( 4 D_B^3 |\Delta|^3 / \pi) - \gamma - 2 \frac{L'(1, \check{\chi}' )}{L(1, \check{\chi}') };\]
here $\gamma$ is the Euler-Mascheroni constant. 
 \end{definition}
%

Finally, we arrive at the definition for the unitary generating series:
\begin{definition} \label{uGenSeriesDef}
Let
\begin{align*} \widehat \Phi^u(\tau) := \widehat\Zed(0, v) \ + \ \frac{1}{4 h(k)} \sum_{m \neq 0} \sum_{[\phi] \in Opt / \OBxone} \ \left( \widehat\Zed(m, \phi,v) \ + 
\ \widehat\Zed{}^*\left( m , \phi, v \right)  \right) q^m,
\end{align*}
where $\tau = u + i v \in \lie{H}$, $q = \exp( 2 \pi i \tau)$, and the sum $\sum_{[\phi]}$ is  over any set of representatives of optimal embeddings taken up to $\OBxone$-conjugacy.
\end{definition}
We shall also need to consider the base change
\[  \widehat\Phi_{/o_k}^{o} (\tau) \ := \ \sum_{n \in \Z} \widehat \Zed{}^o(n, v)_{/o_k} \ q_{\tau}^n \ \in \ \CHR(\CBok) \llbracket q_{\tau}^{\pm 1} \rrbracket ,\]
cf.\ \cite[Theorem 3.6.1]{GilletSoule}. Explicitly, the coefficients are given by 
\[ \widehat \Zed{}^o(n, v)_{/o_k}  \ := \  \left( \Zed{}^o(n, v)_{/o_k} , (Gr^o(n,v)^{\sigma_0}, Gr^o(n,v)^{\sigma_1})  \right), \]
where the two Green functions are obtained by pulling back $Gr^o(n,v)$ via the natural map
 \[ \CBok(\C^{\sigma_i}) \coprod \CBok(\C^{\sigma_1}) \ \to \ \CBZ(\C). \] 

\begin{theorem}[Main theorem] \label{mainThm}
Recall our standing assumptions: $\Delta$ is squarefree and even, and that each prime dividing $D_B$ is inert in $k$.
Then we have an equality of $q$-expansions
\begin{equation} \label{mainThmEqn}
 Sh_{|\Delta|}(\widehat\Phi^o_{/o_k}) (w)  \ = \ \widehat\Phi^u(w). 
\end{equation}

\begin{proof}[Outline of proof]
As discussed in the introduction, the proof begins with the decomposition
\[ \CHR(\CBok)= (Vert \oplus \R \widehat\omega ) \ \oplus \ An \ \oplus \ \widetilde{MW}, \]
together with a collection of `good' spanning vectors for the first two summands, as in \refSec{chowSec}; we need to show that for each component of the orthogonal generating series, the Shimura lift of that component is equal to the corresponding component of the unitary generating series. 
%

First, we observe that the $\widetilde{MW}$-component of an arithmetic cycle depends only on its restriction to the generic fibre. Indeed, 
let $\widehat\Zed = (\Zed, g_{\Zed}) \in \CHR(\CB)_{\R}$, and write its decomposition
\[ \widehat \Zed \ = \ \widehat\Zed_{MW} \ + \ (\widehat\Zed_{\omega} + \widehat\Zed_{Vert}) \ + \ \widehat\Zed_{an} \]
with respect to the previous display.
Pairing against $\mathbf 1 = (0,1) \in Vert$ yields 
\[ \deg \Zed_k \ = \ \la \widehat \Zed, \mathbf 1 \ra \ = \ \la \widehat\Zed_{\omega} , \mathbf 1 \ra \ = \  \deg (\Zed_{\omega} )_k,\]
and so 
\[ \widehat\Zed_{\omega} \ = \ \frac{ \deg \Zed_k}{ \deg\omega_k} \ \widehat\omega. \]
Recall that the `restrict-to-the-generic-fibre' morphism $res_{k}$ induces an isomorphism
\[ res_{k} \colon \widetilde{MW} \isomto  Jac(\CB_{,k})(k) \otimes \R, \qquad (\Zed, g_{\Zed}) \mapsto \Zed_{k}\]
where $Jac(\CB_{,k})(k)$ is the group of rational points of the generic fibre of $\CBok$, cf. \refSec{chowSec}.
In particular, for any class $\widehat\Zed = (\Zed, g_{\Zed}) \in \CHR(\CBok)$, we have
\begin{align}
 res_{k} \widehat\Zed_{MW} \ =\  res_{k} \widehat\Zed \ - \  res_{k} \widehat\Zed_{\omega}  = \ \Zed_{k}  \ -  \ \frac{\deg \Zed_k}{\deg\omega_k} \omega_{k} . \label{resQMWEqn}
\end{align}
Thus, if the generic fibres of two arithmetic divisors are equal, then so are their $\widetilde{MW}$-components. In particular, Corollary 4.11 of \cite{San1} implies the desired Shimura lift relation for the $\widetilde{MW}$ components; the same corollary implies that the Shimura lift relation holds for the vertical components as well. 
%

It remains to  prove the Shimura lift formula for the analytic components and the Hodge component (obtained by pairing against the Hodge class $\widehat\omega$), which appear in  \refSec{anSec} and \refSec{hodgeSec} respectively.
\end{proof}
\end{theorem}

%

\section{Conductors and Frobenius types of orthogonal special cycles} 
In this brief interlude, we prove a key relation between the two indexing sets $\Omega^o(n)$ and $\Omega^{\pm}(m, \lie{a}, \phi)$ that appear in the complex uniformizations of the orthogonal and unitary special cycles, at least in the special case that the squarefree part of $n$ is equal to $|\Delta|$. 

To start, recall that
\[ \Omega^o(n) := \left\{ \xi \in \OB \ | \ Trd(\xi) = 0, \ Nrd(\xi) = n \right\}. \]
Suppose $\xi \in \Omega^o(|\Delta|t^2)$, for some positive integer $t$; it induces an embedding $i_{\xi} \colon k \to B$ determined by the formula
\[ i_{\xi} \left( t \sqrt{\Delta} \right) = \xi. \]
We define \emph{the conductor $c(\xi)$ of $\xi$ } to be the smallest integer $c$ such that 
\[ i_{\xi}( o_c) \subset \OB, \]
where 
\[ o_c \  =  \  \Z [ c \sqrt{\Delta}]   \ \subset \  \ o_k\]
 is the (unique) order of conductor $c$. Note that since $\xi \in \OB$, it follows that $c(\xi)$ divides $t$, and one furthermore has the property $(c(\xi), D_B) = 1$, by valuation considerations. Finally, one easily checks that $c(\xi)$ only depends on the $\OBx$-conjugacy class of $\xi$. 
  
We may also consider the situation at primes dividing $D_B$. Let $p$ be such a prime, and recall our standing assumption that $p$ is inert in $k$. 
As the conductor $c(\xi)$ is relatively prime to $D_B$,  we may pass to the local algebras at $p$ and take quotients in order to obtain an isomorphism 
\[ \overline{i_{\xi,p}} \colon  \ o_{k,p} / (p) \ \isomto  \ \OBp  / (\theta), \]
where $\theta \in \OB$ is some fixed element such that $\theta^2 = -D_B$. Now suppose $\phi \colon o_k \to \OB$ is any embedding, which gives another isomorphism
\[ \overline \phi_p \colon  \ o_{k,p} / (p) \ \isomto  \ \OBp  / (\theta).  \]
As the source and target are both isomorphic to $\F_{p^2}$, there are two possibilities: either $\overline \phi_p = \overline{i_{\xi, p}}$, or they differ by the Frobenius on $\F_{p^2}$. 

It will be convenient to keep track of the various possibilities, over the prime factors of $D_B$, in the following way. We define the \emph{Frobenius type $\nu(\xi, \phi)$ of $\xi$ relative to $\phi$ } to be
\begin{equation}
 \nu(\xi, \phi) \ := \ \prod_{p | D_B} \ \nu_p(\xi, \phi), \qquad \text{where } \quad
\nu_{p}(\xi, \phi) := \begin{cases} 
						1,				&	\text{if } \overline \phi_{p} = \overline {i_{\xi,p}} \\
						p, 				&	\text{otherwise} .
				\end{cases}
\end{equation}

The following lemma, which is a straightforward exercise in quaternion algebras, implies that the Frobenius type is invariant under conjugation by $\OBx$ in both variables: 
\[ \nu( \xi, \phi) \ = \ \nu( \epsilon  \cdot \xi \cdot \epsilon^{-1} , \ \phi) \ = \ \nu(\xi, \ Ad_{\epsilon} \circ \phi ), \qquad \text{for any } \epsilon \in \OBx\]

\begin{lemma}  \label{frobTypeLemma}
Let $p | D_B$, and $\varphi, \varphi' \colon o_{k,p}  \to \OBp$ any two embeddings, with reductions 
\[ \overline \varphi, \overline \varphi' \colon \ o_{k,p} /( p) \  \to \  \OBp / (\theta). \]
Then $\overline \varphi = \overline \varphi'$ if and only if $\varphi = Ad_t \circ \varphi' $ for some $t \in B_p$ with $ord_p Nrd(t)$ even.   \qed
\end{lemma}

All told, for a fixed embedding $\phi$, we obtain a disjoint decomposition
\[ \Omega^o(|\Delta|t^2) \ = \ \coprod_{ \substack{ c | t \\ (c, D_B) = 1 } } \coprod_{\nu | D_B} \Omega^o(|\Delta|t^2, c, \nu; \phi) \]
 where $\Omega^o(|\Delta|a^2,c, \nu; \phi)$ is the  subset of elements $\xi \in \Omega^o(|\Delta|a^2)$ with $c(\xi) = c$ and $\nu(\xi, \phi) = \nu$.

%
The main proposition of this section is a comparison between $\Omega^o(|\Delta|m^2)$ and $\Omega^{\pm}(m, \lie{a}, \phi)$, whose definition we recall:
\[ \Omega^{\pm} (m, \lie{a}, \phi):= \left\{ y \in \phi(\lie{a})^{-1} \OB \ \Big{|}  \ (y,y)_{\phi} =  \Delta Nrd(y) = \pm \frac{m}{ N(\lie{a})}  \right\} \]
There is a natural map
\begin{align*}
\Omega^+(m, \lie{a}, \phi) \ \ \to& \ \  \Omega^o(|\Delta|m^2)  \\
\eta \ \mapsto&  \ \ m \cdot \eta^{-1} \phi(\sqrt{\Delta}) \eta 
\end{align*}
whose image lies in the subset
\[ \Omega^o(|\Delta|m^2)^{\phi-\text{pos}}  \ := \ \left\{ \xi \in \Omega^o(|\Delta|m^2) \ | \ \xi = m b^{-1} \phi(\sqrt{\Delta}) b \ \text{for some } b \in B^{\times} \text{ with }(b,b)_{\phi}>0 \right\}. \]
 
 \begin{proposition} \label{OmegasCompProp}
 Fix a set of representatives $\{ \lie{a}_1, \dots, \lie{a}_h \}$ of $Cl(k)$, and consider the map
 \[ f \colon  \coprod_{i = 1}^{h} \ \Omega^+(m, \lie{a}_i, \phi) \ \to \ \Omega^o(|\Delta|m^2)^{\phi-\text{pos}}, \qquad \eta \mapsto \xi = m \cdot \eta^{-1}  \phi(\sqrt{\Delta}) \eta .\]
 Then for any $\xi \in \Omega^o(|\Delta|m^2)^{\phi-\text{pos}}$, we have
 \[ \#  f^{-1}(\xi) \ = \ |o_k^{\times}| \ \rho \left( \frac{m}{c(\xi) \ \nu(\xi, \phi) \ |\Delta|} \right) \]
 where $\rho(N)$ denotes the number of (integral) $o_k$-ideals of norm $N$. 
 \begin{proof}
 By definition, 
 \[ f^{-1}(\xi) \ = \ \coprod_{i=1}^h  \left\{ b \in \phi(\lie{a}_i)^{-1}\OB \ | \ Nrd(b) = \frac{m}{\Delta N(\lie{a}_i)}, \ \xi = m \cdot b^{-1} \phi(\sqrt{\Delta}) b \right\}. \]
 By assumption, there exists an element $b_0 \in B^{\times}$ such that $\xi = m \cdot b_0^{-1} \phi(\sqrt{\Delta}) b_0$ and $(b_0, b_0)_{\phi} = \Delta Nrd(b_0)>0$. Note that if $b_1 \in B^{\times}$ is another such element, then
 \[ b_1 = \phi(a) b_0 , \qquad \text{for some } a \in k^{\times}. \]
 
 Hence, fixing one such choice $b_0 \in B^{\times}$, we have
 \begin{align}
 \notag \# f^{-1}(\xi) \ =& \ \# \coprod_{i=1}^h  \left\{ a \in k \ | \ N(a) = \frac{m}{\Delta N(\lie{a}_i) Nrd(b_0)}, \ \phi(a) \in \phi(\lie{a}_i)^{-1} \OB b_0^{-1} \right\}  \\
  =& \ |o_k^{\times}| \cdot \# \left\{ \lie{a} \subset k \text{ a fractional ideal} \ | \ N(\lie{a}) = \frac{m}{\Delta Nrd(b_0) }, \ \phi(\lie{a}) \subset \OB b_0^{-1} \right\}. \label{fibreOmegaEqn}
  \end{align}
 We shall analyze the ideals appearing in the above display on a prime-by-prime basis. To start, suppose $p \nmid D_B$. We may fix an isomorphism $\OBp \simeq M_2(\Z_p)$, and without loss of generality, we may assume that $\phi(\sqrt{\Delta})$ is identified with the matrix $(\begin{smallmatrix} & 1 \\ \Delta & \end{smallmatrix})$  - note that the assumption that $\Delta$ is even ensures that this is possible even when $p = 2$. 
 
 Let $c = c(\xi)$ denote the conductor of $\xi$, and set
 \[ w(c)_p := \begin{pmatrix} c & \\  & 1 \end{pmatrix} \ \in \OBp. \]
 One easily checks that the embedding $\phi_c := Ad_{w(c)_p} \circ \phi$ has conductor $c$; in other words
 \[ \phi_c(o_{c,p} ) \ = \ \phi_c(k_p) \cap \OBp. \]
 By assumption, the same is true for the embedding $i_{\xi,p}$. Hence, by \cite[Theoreme 3.2]{Vig}, these two embeddings differ by conjugation by an element of $\mathcal{O}_{B,p}^{\times}$, and so there exist elements $a_p \in k_p$ and $u_p \in \mathcal{O}_{B,p}^{\times}$ such that
 \[ b_0 \ = \ \phi(a_p) \ w(c)_p \ u_p \ \in B_p^{\times}. \]
 
Next, suppose $p \mid D_B$, and so in particular $p$ is inert in $k$. Fix a uniformizer $\Pi \in \OBp$.  We may then write
\[ b_0 = \phi(a_p) \cdot \Pi^{e_p}  \cdot u_p  \in B_p^{\times},\]
where $a_p \in k_p$, $e_p$ is either 0 or 1, and $u \in \mathcal{O}_{B,p}^{\times}$.  By  \refLemma{frobTypeLemma}, we have
\[ e_p \ = \ ord_p \nu_p(\xi, \phi) .\]

Let $\lie{a}_0 := \left( \prod_{p} a_p \cdot o_{k,p} \right) \cap k$ denote the fractional ideal corresponding to the finite adele $(a_p)$. By the product formula, we have
\begin{align*}
|Nrd(b_0)| \ =& \ \prod_{p \nmid D_B} | n(a_p) \cdot Nrd(w(c)_p) |^{-1}_p \ \prod_{p | D_B} | n(a_p) \cdot Nrd(\Pi)^{e_p} |^{-1}_p \\
=& \ N(\lie{a_0}) \prod_{p \nmid D_B} |Nrd(w(c)_p)|^{-1}_p \ \prod_{p|D_B} |Nrd(\Pi)^{e_p}|^{-1}_p \\
=& \ N(\lie{a_0}) \cdot c \cdot \nu(\xi, \phi).
\end{align*}
Thus, replacing the ideals $\lie{a}$ appearing in \eqref{fibreOmegaEqn} by $\lie{a_0} \cdot \lie{a}$, and recalling that $Nrd(b)<0$ by assumption, we have
\begin{align*}
\# f^{-1}(\xi)  \ = \ |o_k^{\times}| \cdot \#  \Big\{ \lie{a} \subset k \ | \ N(\lie{a}) = \frac{m}{|\Delta| c(\xi) \nu(\xi, \eta)  }, \  \phi(\lie{a}  &) \subset \OBp w(c)_p^{-1} \text{ for all } p \nmid D_B,  \\ 
&\text{and } \phi(\lie{a}) \subset \OBp \Pi^{-e_p} \text{ for all } p | D_B \Big\}.
\end{align*}
Finally, it remains to check that the ideals appearing in the above display are in fact integral. Indeed, one checks locally that for such an ideal $\lie{a}$, we have $\phi(\lie{a}) \subset \OBp$ for all primes $p$: when $p | D_B$, this follows from valuation considerations, and when $p \nmid D_B$, this follows from the particular choice of $w(c)_p$. Hence $\phi(\lie{a}) \subset \OB$, and since $\phi$ is optimal, it follows that $\lie{a} \subset o_k$. 
 \end{proof}
 \end{proposition}
 
Similarly, if we set 
 \[ \Omega^o(|\Delta|m^2)^{\phi-\text{neg}}  \ := \ \left\{ \xi \in \Omega^o(|\Delta|m^2) \ | \ \xi = m b^{-1} \phi(\sqrt{\Delta}) b \ \text{for some } b \in B^{\times} \text{ with }(b,b)_{\phi}<0 \right\}, \]
 then we obtain a map
 \[ g \colon \coprod_{[\lie{a}]}  \Omega^-(m, \lie{a}, \phi) \ \to \ \Omega^o(|\Delta|m^2)^{\phi-\text{neg}}, \]
 whose fibres, by the same argument as above, have cardinality 
 \[ \#  g^{-1}(\xi) \ = \ |o_k^{\times}| \ \rho \left( \frac{m}{c(\xi) \ \nu(\xi, \phi) \ |\Delta|} \right). \]
 We also observe in passing that by the Noether-Skolem theorem, we have
 \[ \Omega^o(|\Delta|m^2)  \ = \ \Omega^o(|\Delta|m^2)^{\phi-\text{pos}} \ \coprod \ \Omega^o(|\Delta|m^2)^{\phi-\text{neg}}. \]

The following calculation will appear several times in the sequel, and so will be worthwhile to state on its own:
\begin{lemma} \label{rhoLemma}
For any $m>0$, we have
\[   \sum_{\nu | D_B} \ \rho \left( \frac{ m}{\nu} \right) \ = \ \sum_{a | m} \ \chi'(a), \]
where  $\chi'$ is the character appearing in \eqref{chi'DefEqn}.
\begin{proof}
First, we claim that for any integer $N >0$, we have 
\[ \rho(N) \ = \ \sum_{a | N} \chi_k(a).\]
Indeed, both sides are multiplicative, so it suffices to verify this formula when $N$ is a power of a prime $p$. The formula is easily verified by considering separately the cases where $p$ is inert, split, or ramified in $k$. 

Next, we note that if $ord_pN $ is even for every prime $p | D_B$, then
\[ \sum_{a | N} \chi_k(a) \ = \ \sum_{\substack{ a| N  \\ (a, D_B) = 1}} \chi_k(a) \ = \ \sum_{a | N} \chi'(a). \]
Finally, we observe that for a given $m$, there is a \emph{unique} $\nu^* | D_B$ such that $ord_p(m/ \nu^*)$ is even for all $p|D_B$. If $\nu \neq \nu^*$, then $\rho( m / \nu) = 0$, because an inert prime occurs to an odd power in $m/\nu$. 
Thus
\[ \sum_{\nu | D_B} \rho( m / \nu) \ = \ \rho (m / \nu^*) \ = \sum_{a | (m / \nu^*)} \ \chi'(a) \ = \ \sum_{a | m} \chi'(a). \]
\end{proof}
\end{lemma}

\section{Proof of main theorem: Analytic components} \label{anSec}
In this section, we prove the Shimura lift formula for the `analytic' components of our generating series. We begin by briefly recalling how these components are defined. Let 
$\mathbf f = (f^{\sigma_0}, f^{\sigma_1})$ denote a pair smooth functions on $\CBok(\C^{\sigma_0})$ and $\CBok(\C^{\sigma_1})$ respectively, such that 
\begin{equation} \label{grFnSymEqn}
f^{\sigma_1} \ = \ \iota^* f^{\sigma_0} 
\end{equation}
and
\begin{equation} \label{fnOrthToHodgeEqn}
 \int_{\CB(\C^{\sigma_i})} f^{\sigma_i} \cdot c_1( \widehat\omega)^{\sigma_i} = 0, 
\end{equation}
where $c_1(\widehat\omega)^{\sigma_i}$ is the first Chern form of the Hodge class $\widehat{\omega}$. 
We obtain a class $\hat{\mathbf f} := (0, \mathbf f) \in \CHR(\CBok)$. If $ \widehat\Zed = (\Zed, g({\Zed}))$ is any arithmetic class, then
\begin{equation*} 
 \la \widehat\Zed,  \hat{\mathbf f } \ra \ = \ \frac{1}{2} \sum_{i=0}^1 \int_{\CB(\C^{\sigma_i})}  f^{\sigma_i} \cdot \omega^{\sigma_i}({\Zed}), 
\end{equation*}
where $\omega^{\sigma_i}(\Zed)$ is the smooth $(1,1)$-form on $\CB(\C^{\sigma_i})$ such that
\[ d d^c g^{\sigma_i}(\Zed) + \ \delta_{\Zed(\C^{\sigma_i})} \ = \ [ \omega^{\sigma_i}(\Zed) ]. \]
As  $g^{\sigma_1}(\Zed) = \iota^* g^{\sigma_0}(\Zed)$ we find that
\begin{equation} \label{anPairEqn}
\la \widehat\Zed, \hat{\mathbf f }\ra \ = \  \int_{\CB(\C^{\sigma_0})}  f^{\sigma_0} \cdot \omega^{\sigma_0}({\Zed}).
\end{equation}

Applying this discussion first to the orthogonal generating series we have that for $\tau = u + iv \in \lie{H}$, 
\begin{equation}
\la \widehat\Phi_{/o_k}^o(\tau), \hat{\mathbf f} \ra \ =  \  \sum_{\substack{n \in \Z \\ n \neq 0}}  \ \left(\int_{\CB(\C)} \ f^{\sigma_0}   \cdot \omega^{\sigma_0}( {\Zed^o(n,v)} )\right) \ q^n.
\end{equation}
We recall  $\omega^{\sigma_0}\left({\Zed^o(n,v)}\right)$ has the following explicit form, cf. \eqref{oSCgrEqn},
\[ \omega^{\sigma_0}\left(\Zed^o(n,v)\right) \ = \ \psi^o(n,v) \cdot c_1(\widehat\omega),\]
in terms of the uniformization $\CB(\C^{\sigma_0}) \ = \ [ \OBxone \backslash \lie{H} ] $.
Here $\psi^o(n,v)(z)$ is the $\OBxone$-invariant function on $\lie{H}$ given by
\begin{equation} \label{psioDefEqn}
 \psi^o(n, v)(z)  \ = \ \sum_{x \in \Omega^o(n)} \ \left( 4 \pi v \left\{ R^o(x, z) + 2 Nrd(x) \right\}  - 1 \right) \ e^{-2 \pi v R^o(x,z)}, 
\end{equation}
where $R^o(\cdot, z)$ is the majorant determined by $z$ as in \eqref{RoDefEqn}. 

Turning now to the unitary generating series, recall that for $m \in \Z$ with $m \neq 0$ and $y \in \R_{>0}$, we had defined classes $\widehat\Zed(m, \phi, \eta)$ -- and their re-scaled variants $\widehat\Zed{}^*(m,\phi,\eta)$, cf. \eqref{Z*DefEqn} -- where the corresponding smooth $(1,1)$-form $\omega^{\sigma_0}(\widehat\Zed(m, \phi, \eta))$ can be written as follows:
\[ \omega^{\sigma_0}\left(\widehat\Zed(m, \phi, \eta ) \right)= \frac{1}{|o_k^{\times}|}\psi^u(m, \phi, \eta) \cdot c_1(\widehat\omega ). \]
Here, for $z \in \lie{H}$ and $\eta \in \R_{>0}$, we recall that
\begin{align}
 \psi^u(m, \phi, \eta) \ =& \frac{1}{2}  \sum_{[\lie{a} ] \in Cl(k)} \Bigg\{ \sum_{y \in  \Omega^+(m, \lie{a}, \phi)} \Bigg[2 \pi N(\lie{a}) \eta \Big( R_{\phi}(  y, z) +2 \Delta Nrd( y) \Big) -1 \Bigg] e^{-2 \pi N(\lie{a}) \eta R_{\phi}(\mu y, z)} \notag  \\
 & \qquad  \  +  \sum_{y \in  \Omega^-(m, \lie{a}, \phi)} \Bigg[2 \pi N(\lie{a}) \eta R_{\phi}( y, z) -1 \Bigg] e^{-2 \pi N(\lie{a}) \eta(R_{\phi}( y, z)  + 2 \Delta Nrd( y))} \Bigg\} 
 \label{psiUDefEqn}.
 \end{align}
 
By \eqref{fnOrthToHodgeEqn}, we have $\la \widehat\Zed(0, v) ,  \widehat {\mathbf f} \ra = 0$, and so
\[ \la \widehat\Phi{}^u (\tau), \widehat{ \mathbf f} \ra \ = \frac{1}{4 h(k) }  \sum_{\substack{m \in \Z \\ m \neq 0}}  \ \sum_{[\phi]}  \ \left( \int_{\CB(\C)} f^{\sigma_0} \cdot \left\{ \omega^{\sigma_0} \left( {\widehat\Zed(m,\phi, v)} \right) + \omega^{\sigma_0}\left( \widehat\Zed{}^*(m ,\phi, v \right) \right\} \right) q^m. \]


\begin{lemma} \label{genFibreUnitCorLemma} Suppose $\phi \colon o_k \to \OB$ is an optimal embedding, $m \in \Z$, and $\mu \in \OB$ such that $Nrd(\mu) = (m, D_B)$. Set $\phi' :=  Ad_{\mu^{-1}} \circ \phi$, and note that $\phi'$ is again optimal. Then
	\begin{enumerate}[(i)]
		\item $\Zed^*(m, \phi')_{k} \ = \ \Zed(m, \phi)_{k}$ and
		\[ \sum_{[\phi] \in Opt / \OBxone} \Zed(m, \phi)_{k} \  = \ \sum_{[\phi]} \ \Zed^* \left( m, \phi \right)_{k} \]
		as cycles on $\CB_{/k}$;
		\item  $Gr^{u*}(m, \phi', \eta) = Gr^u(m, \phi, \eta)$;
		\item and $\omega^{\sigma_0}\left( \ \widehat\Zed{}^*(m , \phi', \eta )\  \right) \ = \ \omega^{\sigma_0}\left( \  \widehat\Zed(m,  \phi,\eta )\ \right)$.
	\end{enumerate}
\begin{proof}

(i) For any fractional ideal $\lie{a}$, left-multiplication by $\mu$ gives a bijection
\[ \Omega^{\pm}(m/ (m,D_B), \lie{a}, \phi') \isomto \Omega^{\pm} \left( m, \ \lie{a}, \ \phi \right), \qquad y \mapsto \mu \cdot y.\]
Furthermore, the map 
\[ \D(B^{\phi'}_{\R}) \isomto \D(B^{\phi}_{\R}) \qquad \zeta \mapsto \mu \zeta \]
is an isomorphism that identifies $\D_{y} = y^{\perp} \in \D(B^{\phi'}_{\R})^- $ with $\D_{\mu \cdot y} = (\mu y )^{\perp} \in \D(B^{\phi}_{\R})^-$ when $ y \in \Omega^+(\dots)$, and identifies $\D_{y} = span_{\phi'(k_{\R})}(y) $ with $D_{\mu y} = span_{\phi(k_{\R}})(\mu y)$, when $y \in \Omega^-(\dots)$ . The first claim in the lemma then follows immediately from the complex uniformizations. 
The second claim follows by noting that the assignment $\phi \mapsto Ad_{u^{-1}} \circ \phi $ acts as a permutation on the set of $\OBxone$ classes of optimal embeddings.


(ii) Set $d := (|m|, D_B)$ for notational simplicity.
Recall that for a point $z \in \lie{H}$ and a vector $b \in B_{\R}$, we had defined the majorant
\[ R_{\phi}(b, z) \ :=  \ - 2( pr_{\zeta(z)}(b), \  pr_{\zeta(z)} (b) )_{\phi},  \]
where $\zeta(z) := \{ v \in B_{\R} \ |\  v J_z = - \phi(\mathbb I_k) v \}$
is the $\phi(k_{\R})$-stable negative line in $B_{\R}$ corresponding to $z$, and $pr_{\zeta(z)}$ is the orthogonal projection onto this line with respect to the hermitian form $(\cdot, \cdot)_{\phi}$ defined in \eqref{hermFormEqn}. 

A direct calculation reveals that for $\phi' = Ad_{\mu^{-1}} \circ \phi$, we have
\[ R_{\phi'} (b, z ) \ = \ Nrd(\mu^{-1}) R_{\phi}(\mu \cdot b, z) \ = \ d^{-1} R_{\phi}(\mu \cdot b , z). \]
Recall also that for any fractional ideal $\lie{a}$, we have a bijection
\[ \Omega^{\pm}(m / d, \lie{a}, \phi') \isomto \Omega^{\pm} (m, \lie{a}, \phi), \qquad y \mapsto \mu \cdot y.  \]
Thus
\begin{align*}
Gr^{u*}(m, \phi', \eta)(z) \ =&  \ \frac{1}{|o_k^{\times}|}  \sum_{[\lie{a}] \in Cl(k)} \Big[ \sum_{y \in \Omega^{+}(m/d, \lie{a}, \phi')}  \beta_1 \left( \frac{2 \pi N(\lie{a}) \eta}{d} \cdot R_{\phi'}(y, z) \right) \\
 &  \qquad \qquad +  \sum_{y \in \Omega^{-}(m/d, \lie{a}, \phi')}  \beta_1\left(\frac{ 2 \pi N(\lie{a}) \eta}{d} \cdot ( R_{\phi'}(y, z) + 2\Delta Nrd(y) ) \right) \Big]   \\
  =&  \ \frac{1}{|o_k^{\times}|}  \sum_{[\lie{a}] \in Cl(k)} \Big[ \sum_{y \in \Omega^{+}(m/d, \lie{a}, \phi')}  \beta_1\big( 2 \pi N(\lie{a}) \eta \cdot R_{\phi}( \mu y, z) \big) \\
 &  \qquad \qquad +  \sum_{y \in \Omega^{-}(m/d, \lie{a}, \phi')}  \beta_1\big( 2 \pi N(\lie{a}) \eta \cdot ( R_{\phi}( \mu y, z) + 2\Delta Nrd( \mu y) ) \big) \Big]  \\
=& \ Gr^u(m, \phi, \eta),
\end{align*}
as required.

(iii) follows immediately from (i) and (ii), via Green's equation \eqref{GrGenEqn}. 
\end{proof}
\end{lemma}

After summing over $\OBxone$-classes of optimal embeddings, the $\hat{\mathbf f}$-component of the unitary generating series becomes (for $w = \xi + i \eta \in \lie{H}$)
\begin{equation}
\la \widehat\Phi{}^u (w), \widehat{\mathbf f }\ra \ = \ \frac{1}{2 h(k) }  \sum_{\substack{m \in \Z \\ m \neq 0}}  \ \sum_{[\phi]}  \ \left( \int_{\CB(\C^{\sigma_0})} f^{\sigma_0} \cdot  \omega^{\sigma_0} \left( {\widehat\Zed(m,\phi, \eta)} \right) \right) q_w^m .
\end{equation}
The main theorem of this section equates the Shimura lift of the $f$-component of $\widehat\Phi^o$ with that of $\widehat\Phi^u$. In fact, we shall prove something a priori slightly stronger, namely that the $\widehat{\mathbf f}$-components of these two generating series arise by integrating $f^{\sigma_0}$ against certain theta functions, and that these theta functions are related by the Shimura lift. 

Let 
\begin{equation} \label{ThetaOEqn}
 \varTheta^o(\tau,z)\  :=  \ \sum_{\substack{ n \in \Z \\ n \neq 0}}  \ \psi^o(n,v)(z) \ q_{\tau}^n , \qquad \tau = u + i v
\end{equation}
and 
\begin{equation*} 
 \varTheta^u(w, z) \ := \ \frac{1}{2 h(k)|o_k^{\times}|} \ \sum_{\substack{m \in \Z \\ m \neq 0 }} \ \sum_{[\phi]} \psi^u(m,  \phi, \eta)(z) \ q_w^m, \qquad w = \xi + i \eta.
\end{equation*}
The function $\varTheta^o(\tau, z)$ can be shown to be a theta kernel arising via the Weil representation for the dual pair $(O(1,2), SL_2)$, cf.\ \cite[(4.4.5)]{KRYbook}. In particular, it is a modular form in the $\tau$-variable of weight 3/2. Though we will not pause to do so here, it can be also be shown that $\varTheta^u$ arises via restriction of a theta kernel attached to $(O(2,2), SL_2)$ of weight 2. 

Since the sets $\Omega^{\pm}( m, \lie{a}, \phi)$ appearing in \refProp{cpxUnifProp} are empty unless $|\Delta| $ divides $m$, 
\begin{align} 
 \varTheta^u(w, z) \ =& \ \frac{1}{2 h(k)|o_k^{\times}|} \ \sum_{\substack{m \in \Z \\ m \neq 0 }} \ \sum_{[\phi]} \psi^u(|\Delta|m,  \phi, \eta)(z) \ q_w^{|\Delta|m}  \notag \\
 =& \ \frac{1}{2 h(k)|o_k^{\times}|} \  \sum_{\substack{m \in \Z \\ m \neq 0 }} \left( \sum_{[\phi]} \psi^u(|\Delta|m,  \phi, \eta)(z)\ e^{-2 \pi |\Delta| m \eta} \right) \ e^{2 \pi i |\Delta| m \xi} \label{ThetaUEqn}.
 \end{align}
As the sums defining $\varTheta^o$ and $\varTheta^u$ are absolutely convergent, we have
\begin{equation*}
\la \widehat\Phi^o(\tau), \hat{\mathbf f} \ra  = \int_{\CB(\C^{\sigma_0})} f^{\sigma_0}(z) \cdot \varTheta^o(\tau, z) c_1(\widehat\omega)
\end{equation*}
and
\begin{equation*}
\la \widehat\Phi^u(w) , \hat{\mathbf f} \ra = \int_{\CB(\C^{\sigma_0})} f^{\sigma_0}(z) \cdot \varTheta^u(w, z) c_1(\widehat\omega).
\end{equation*}

\begin{theorem} \label{anCompThetaThm} Suppose $\Delta<0$ is squarefree and even, and every prime dividing $D_B$ is inert in $k$. Then
\[ Sh_{|\Delta|} \varTheta^o(\cdot, z) (w) \ = \ \varTheta^u(w,z) \]
\begin{proof}
As $\varTheta^o(\tau, z)$ is a non-holomorphic form in the $\tau$ variable (of weight 3/2, level $\Gamma_0(4 D_B)$, and trivial character), its Shimura lift is given by integrating against the Niwa-Shintani theta kernel $\Theta^{\#}_{NS}$  -- in the notation of \refSec{shLiftSec}, we have $N =  D_B, \lambda = 1, t = |\Delta|$ and $\chi_t = \chi'$:
\begin{align*}
Sh_{|\Delta|} \varTheta^o(\cdot, z)(w) \ =& \ C \cdot \int_{\Gamma_0( 4 D_B |\Delta| ) \backslash \lie{H} } v^{3/2} \  \varTheta^o_{|\Delta|}(\tau, z) \ \overline{ \Theta^{\#}_{NS}(\tau, w)} \ d  \mu(\tau) 
\end{align*} 
where $C = -2(|\Delta|D_B)^{-3/4}$ and the measure $d \mu(\tau) = du \cdot dv / v^2$ is the standard hyperbolic measure on $\lie{H}$.
 Substituting the Poincar\'e series expansion for $\Theta^{\#}_{NS}$ of  \refProp{thetaPoincare} (formally for now, we will justify this step momentarily), and applying the usual `unfolding' trick, we have 
 \begin{align}
Sh_{|\Delta|} \varTheta^o(\cdot, z)(w) \ =& \   \frac{1}{8} \int_{ \Gamma_{\infty} \backslash \lie{H}} v^{3/2} \ \varTheta^o(|\Delta| \tau, z)  \Bigg[ \sum_{m \in \Z} \chi'(m) \sum_{\mu = 0}^1 4^{\mu} \left( \frac{\eta m}{v} \right)^{1-\mu} \notag \\
& \qquad \times \exp \left( - \frac{\pi \eta^2 m^2}{4v} \right) \overline{\theta_{\mu}(\tau, - \xi m /2)} \Bigg]  d\mu(\tau) \notag \\
=& \   \frac{1}{8}  \int_0^{\infty}   v^{-1/2} \ \exp \left( - \frac{\pi \eta^2 m^2}{4v} \right)  \Bigg[ \sum_{m \in \Z} \chi'(m) \sum_{\mu = 0}^1 4^{\mu} \left( \frac{\eta m}{v} \right)^{1-\mu} \Bigg]  \notag \\ 
& \qquad \times \left(   \int_{0}^1 \varTheta^o(|\Delta| \tau, z) \ \overline{\theta_{\mu}(\tau, - \xi m /2)} \ du \right) \ dv  \label{ShAnalEqn1}
\end{align}
 where 
\[ \theta_{\mu}(\tau, - \xi m/2) = \sum_{\ell \in \Z} \ell^{\mu} e^{2 \pi i \left( \tau \ell^2 - \xi \ell m \right)} \]
for $\mu = 0,1$. A straightforward estimate reveals that \eqref{ShAnalEqn1} is absolutely convergent, and so in particular our interchanges of summations and integrals is justified. 

Computing first the intergral on $u$, we obtain
\begin{align*}
 \int_0^1 \varTheta^o(|\Delta| & (u + i v), z) \  \overline{\theta_{\mu}  (u + iv, - \xi m /2)} du  \\
 =& \sum_{\ell \in \Z} \ \ell^{\mu} \ e^{-2 \pi v \ell^2 + 2 \pi i \xi \ell m} \ \int_0^1 \varTheta^o(|\Delta| (u + i v), z) \ e^{-2 \pi i \ell^2 u} \ du \\
  =& \sum_{\ell \in \Z} \ \ell^{\mu} \ e^{-4 \pi v \ell^2 + 2 \pi i \xi \ell m} \  \begin{cases} \psi^o(|\Delta|n, |\Delta|v)(z), & \text{if } \ell^2 = |\Delta| n \text{ for some } n \neq 0 \\ 0, & \text{otherwise}. \end{cases}
\end{align*}
Since $|\Delta|$ is squarefree, only those $\ell$ with $|\Delta|$ dividing $\ell$ will contribute to the previous display. Thus
\begin{align*} 
Sh_{|\Delta|} \varTheta^o(\cdot, z)(w) = \frac{1}{8} \int_0^{\infty} v^{3/2}  \sum_{m \in \Z}  \chi'(m) \sum_{\mu = 0}^1 4^{\mu} \left( \frac{\eta m}{v} \right)^{1-\mu}  \exp \left( - \frac{\pi \eta^2 m^2}{4v} \right)  \\
\times \sum_{\substack{\ell \in \Z\\ \ell \neq 0}} \ (|\Delta|\ell)^{\mu} \ e^{-4 \pi v |\Delta|^2 \ell^2 + 2 \pi i \xi |\Delta|\ell m}  \  \psi^o(|\Delta|\ell^2, |\Delta|v)(z)  \ \frac{dv}{v^2}.
\end{align*}
Exchanging the summations with the integral, we obtain
\begin{align}
Sh_{|\Delta|} \varTheta^o(\cdot, z)&(w) \ = \ \frac{1}{8} \sum_{m,\ell \in \Z}  \chi'(m)  \int_0^{\infty} v^{-1/2} \Big[  \sum_{\mu = 0}^1 (|\Delta|\ell)^{\mu} 4^{\mu} \left( \frac{\eta m}{v} \right)^{1-\mu}  \Big] \notag \\
&\times  \exp \left( - \frac{\pi \eta^2 m^2}{4v}  -4 \pi v |\Delta|^2 \ell^2 \right) \  \psi^o(|\Delta|\ell^2, |\Delta|v)(z)  \ dv \ \cdot e^{2 \pi i m \ell |\Delta | \xi}  \notag\\
 = \ \frac{1}{4}  & \sum_{\substack{\ell \in \Z \\ \ell \neq 0}} \ \sum_{0< m \mid  |\ell|} \chi'(m)   \int_0^{\infty} v^{-1/2} \Big[  \sum_{\mu = 0}^1 \left(\frac{|\Delta|\ell}{m}\right)^{\mu} 4^{\mu} \left( \frac{\eta m}{v} \right)^{1-\mu}  \Big] \notag \\
&\times  \exp \left( - \frac{\pi \eta^2 m^2}{4v}  - \frac{4 \pi v |\Delta|^2 \ell^2}{m^2}  \right) \  \psi^o(|\Delta|\ell^2/m^2, |\Delta|v)(z)  \ dv \ \cdot e^{2 \pi i \ell |\Delta | \xi} \notag \\
= \   & \sum_{\substack{\ell \in \Z \\ \ell \neq 0}} \ \left( \sum_{0< m \mid  |\ell|} \chi'(m)  \cdot  I^o(\ell, m, \eta)(z) \right) \cdot e^{2 \pi i \ell |\Delta| \xi},  \label{ShLiftAnalFCoeffEqn}
\end{align}
where 
\begin{align*} I^o(\ell, m, \eta)(z) \ :=&  \ \frac{1}{4}  \int_0^{\infty} v^{-1/2} \Big[  \sum_{\mu = 0}^1 \left(\frac{|\Delta|\ell}{m}\right)^{\mu} 4^{\mu} \left( \frac{\eta m}{v} \right)^{1-\mu}  \Big] \\
& \qquad  \times  \exp \left( - \frac{\pi \eta^2 m^2}{4v}  - \frac{4 \pi v |\Delta|^2 \ell^2}{m^2}  \right) \  \psi^o(|\Delta|\ell^2/m^2, |\Delta|v)(z)  \ dv  \\
=& \ \frac{1}{4} \sum_{x \in \Omega^o(|\Delta| \ell^2 / m^2) }  \int_0^{\infty} v^{-1/2}  \Big[  \sum_{\mu = 0}^1 \left(\frac{|\Delta|\ell}{m}\right)^{\mu} 4^{\mu} \left( \frac{\eta m}{v} \right)^{1-\mu}  \Big] \\ 
&  \qquad \times \exp \left( - \frac{\pi \eta^2 m^2}{4v}  - 2 \pi|\Delta| v \left( R^o(x, z) + 2 Nrd(x) \right) \right) \\
&  \qquad \  \times \left\{ 4 \pi |\Delta| v [ R^o(x, z) + 2 Nrd(x) ] - 1 \right\} dv.
\end{align*}
Now for each $x$, this integral is a sum of Bessel $K$-functions of half-integral weight, which can be evaluated explicitly with the aid of the tables of Gradshteyn and Ryzhik, cf.\ \cite[(8.342.(vi)) and (8.468)]{GR}. In particular, for $a,b \in \R_{>0}$, we have the following relations:
\begin{equation}
\label{bessK-3/2} \int_0^{\infty} v^{-3/2} e^{ - \frac{\pi}{2} \left( av + b/v \right)} dv = \sqrt{ \frac{2}{b}} e^{- \pi \sqrt{ab} }
\end{equation}
\begin{equation}
\label{bessK-1/2} \int_0^{\infty} v^{-1/2} e^{ - \frac{\pi}{2} \left( av + b/v \right)} dv = \sqrt{ \frac{2}{a}} e^{- \pi \sqrt{ab} } ,
\end{equation}
\begin{equation}
\label{bessK+1/2}\int_0^{\infty} v^{1/2} e^{ - \frac{\pi}{2} \left( av + b/v \right)} dv =  \frac{\sqrt{2} ( 1+ \pi \sqrt{ab} )} {\pi a^{3/2}} e^{- \pi \sqrt{ab}}.
\end{equation}
After applying these formulae, we arrive at the expression
\begin{align*}
 I^o(m, \ell, \eta)(z) &= \frac12 \sum_{x \in \Omega^o(|\Delta|\ell^2/m^2)} \left( \pi \eta m \left( \sqrt{2|\Delta|} \left\{ R^o(x,z) + 2 Nrd(x) \right\}^{1/2} + 2 |\Delta| \frac{\ell}{m} \right) - 1 \right) \\
 & \qquad \times \exp \left( - \sqrt{2 |\Delta| } \pi  m  \left\{ R^o(x, z) + 2 Nrd(x) \right\}^{1/2} \eta \right).
 \end{align*}
Replacing $x$ by $m\cdot x$ in the preceding sum then yields
\begin{align*}
 I^o(m, \ell, \eta)(z) &=  \sum_{\substack{ x \in \Omega^o(|\Delta|\ell^2) \\ c(x)  \mid (\ell / m)}} \frac12 \left( \pi \eta  \left( \sqrt{2|\Delta|} \left\{ R^o(x,z) + 2 Nrd(x) \right\}^{1/2} + 2 |\Delta| \ell\right) - 1 \right) \\
 & \qquad \times \exp \left( - \sqrt{2 |\Delta| } \pi   \left\{ R^o(x, z) + 2 Nrd(x) \right\}^{1/2} \eta \right) \\
 & =:  \  \sum_{\substack{ x \in \Omega^o(|\Delta|\ell^2) \\ c(x)  \mid (\ell / m)}} I^o(\ell, x, \eta)(z).
 \end{align*}
Now, employing  \refLemma{rhoLemma}, we have
\begin{align} 
\sum_{m | |\ell|} \chi'(m) I^o( \ell, m, \eta)(z) \ =& \ \sum_{m | |\ell|} \chi'(m)  \sum_{\substack{ x \in \Omega^o(|\Delta|\ell^2) \\ c(x)  \mid (\ell / m)}} I^o(\ell, x,\eta)(z)  \notag \\
=& \ \sum_{c| |\ell|} \left(  \sum_{m | |\ell|/c} \chi'(m) \right)\sum_{\substack{ x \in \Omega^o(|\Delta|\ell^2) \\ c(x) = c}} I^o(\ell, x,\eta)(z)\notag \\
=& \sum_{c| |\ell|} \left( \sum_{\nu | D_B} \rho \left( \frac{ |\ell|}{\nu c} \right) \right)\sum_{\substack{ x \in \Omega^o(|\Delta|\ell^2) \\ c(x) = c}} I^o(\ell, x,\eta)(z)\notag \\
=& \ \frac{1}{2 h(k)|o_k^{\times}|} \sum_{[\phi]} \sum_{c| |\ell|} \ \sum_{\nu | D_B} |o_k^{\times}|\rho \left( \frac{ |\Delta\ell|}{\ |\Delta| \nu c} \right) \sum_{\substack{ x \in \Omega^o(|\Delta|\ell^2) \\ c(x) = c \\ \nu(x, \phi) = \nu}} I^o( \ell, x,\eta)(z). \label{shAnalCoeffEqn}
\end{align}
Note that for each $x$ occuring in the above sum,  the term $|o_k^{\times}|\rho( \cdot )$ that precedes it counts the number of elements $y \in \coprod_{[\lie{a}]} \Omega^{\pm}(|\Delta|\ell, \lie{a}, \phi)$ such that 
\[ x = |\ell| \cdot y^{-1} \phi(\sqrt{\Delta} ) y , \]
as in  \refProp{OmegasCompProp}.

%

The following lemma, whose  \hyperref[majProofLemma]{proof}  appears at the end of this section, provides the key link between the majorants associated to $x$ and $y$ in this situation:
\begin{leftbar}[0.85\linewidth]
\begin{lemma} \label{majorantLemma}
Suppose $ y \in B_{\R}^{\times} $ and 
\[ x = t \cdot y^{-1} \phi(\sqrt{\Delta}) y\]
for some $t \in \R^{\times}$, and embedding $\phi \colon k \to B$. Then for any $z \in \lie{H}$, we have
\[ R^o(x, z) \ + \ 2 Nrd(x) \ = \ \frac{2 t^2}{|\Delta| Nrd(y)^2} \ \left( R_{\phi}(y, z) \ + \ \Delta Nrd(y) \right)^2. \]\qed
\end{lemma} 
\end{leftbar}
\noindent We first consider the case $\ell >0$. Choose $y \in \Omega^+(\ell |\Delta|, \lie{a}, \phi)$, so that  
\[ Nrd(y) = \frac{|\Delta| \ell}{\Delta N(\lie{a})}  = - \frac{\ell}{N(\lie{a})}. \]
If we set $x  = \ell \cdot  y^{-1} \phi(\sqrt{\Delta}) y \in \Omega^o(|\Delta|\ell^2)^{\phi-\text{pos}}$, then applying the lemma, we have 
\begin{align*}
 I^o(\ell, x)(z) =& \frac12 \left( \pi \eta  \left( \sqrt{2|\Delta|} \left\{ \frac{\sqrt{2} N(\lie{a}) }{|\Delta|^{1/2}} \cdot \left( R_{\phi}(y,z) + \Delta Nrd(y) \right) \right\} + 2 N(\lie{a})\Delta Nrd(y)\right) - 1 \right) \\
 & \qquad \times \exp \left( - \sqrt{2 |\Delta| } \pi   \left\{ \frac{\sqrt{2} N(\lie{a}) }{|\Delta|^{1/2}} \cdot \left( R_{\phi}(y,z) + \Delta Nrd(y) \right) \right\}  \eta \right) \\
 =&\ \frac12 \  \Big\{ 2 \pi N(\lie{a}) \eta  \left[ R_{\phi}(y,z) +2 \Delta Nrd(y) \right] - 1 \Big\}  \  e^{\ - 2 \pi N(\lie{a}) \eta  \left[ R_{\phi}(y,z) + \Delta Nrd(y) \right ]}.
\end{align*}
Similarly, if we choose $y \in \Omega^-(\ell |\Delta|, \lie{a}, \phi)$, and set $x = \ell \cdot y^{-1} \phi(\sqrt{\Delta}) y  \in \Omega^o(|\Delta|\ell^2)^{\phi-\text{neg}}$, 
then we obtain
\begin{align*}
 I^o(\ell, x)(z) =& \frac12 \left( \pi \eta  \left( \sqrt{2|\Delta|} \left\{ \frac{\sqrt{2} N(\lie{a}) }{|\Delta|^{1/2}} \cdot \left( R_{\phi}(y,z) + \Delta Nrd(y) \right) \right\} - 2 N(\lie{a})\Delta Nrd(y)\right) - 1 \right) \\
 & \qquad \times \exp \left( - \sqrt{2 |\Delta| } \pi   \left\{ \frac{\sqrt{2} N(\lie{a}) }{|\Delta|^{1/2}} \cdot \left( R_{\phi}(y,z) + \Delta Nrd(y) \right) \right\}  \eta \right) \\
 =&\ \frac12 \  \Big\{ 2 \pi N(\lie{a}) \eta R_{\phi}(y,z)  - 1 \Big\}  \  e^{\ - 2 \pi N(\lie{a}) \eta  \left[ R_{\phi}(y,z) + \Delta Nrd(y) \right ]}.
\end{align*}

Returning now to \eqref{shAnalCoeffEqn}, we may replace the sum on $x$ by the sum over elements $y \in \Omega^{\pm}( \dots) $ to obtain (for $\ell > 0$)
\begin{align*}
\sum_{m | |\ell|}  & \chi'(m) I^o( \ell, m)(z) \ = \ \frac{1}{2 h(k) |o_k^{\times}|} \sum_{[\lie{a}]} \sum_{[\phi]}  \\
&\times \Bigg\{  \sum_{y \in \Omega^+(|\Delta| \ell, \lie{a}, \phi) } \frac12 \  \Big\{ 2 \pi N(\lie{a}) \eta  \left[ R_{\phi}(y,z) +2 \Delta Nrd(y) \right] - 1 \Big\}  \  e^{\ - 2 \pi N(\lie{a}) \eta  \left[ R_{\phi}(y,z) + \Delta Nrd(y) \right ]} \\
& \qquad + \sum_{y \in \Omega^-(|\Delta| \ell, \lie{a}, \phi) } \frac12 \  \Big\{ 2 \pi N(\lie{a}) \eta R_{\phi}(y,z)  - 1 \Big\}  \  e^{\ - 2 \pi N(\lie{a}) \eta  \left[ R_{\phi}(y,z) + \Delta Nrd(y) \right ]} \Bigg\} \\
& \qquad \qquad \qquad =  \ \frac{1}{2 h(k) |o_k^{\times}|} \ \psi^u(|\Delta|\ell, \eta)(z)  \ e^{-2 \pi |\Delta| \ell \eta}.
\end{align*}
Hence, comparing \eqref{ShLiftAnalFCoeffEqn} and \eqref{ThetaUEqn}, we see that if $\ell>0$, then the $|\Delta|\ell$'th Fourier coefficients of $Sh_{|\Delta|}\varTheta^o$ and $\varTheta^u$ are equal. A nearly identical calculation reveals that the same holds true when $\ell <0$, concluding the proof of the theorem. 
\end{proof}
\end{theorem}

\begin{proof}[Proof of \refLemma{majorantLemma} ]  \phantomsection \label{majProofLemma}
We first  show that the lemma holds for a specific point $z_0 \in \lie{H}$. Given the embedding $\phi \colon k \to B$, we may write 
\[ B = \phi(k) \ + \ \theta \phi(k)\]
 for some element $\theta $ such that $\theta^{\iota} = - \theta$ and $\theta \phi(a) = \phi(a') \theta$ for all $a \in k$. Note that $\theta^2>0$ since $B$ is indefinite. If we write 
 \[ y = \phi(a) + \theta \phi(b) \]
   with $a, b \in k_{\R}$, then
    \begin{align*}
  x \ =& \ t \cdot y^{-1} \phi(\sqrt{\Delta}) y \\
   =& \ \frac{t}{Nrd(y)} \left( \phi(\sqrt{\Delta})(N(a) + \theta^2 N(b))  \  - \ 2 \theta \phi(a b \sqrt{\Delta}) \right).
\end{align*}
Let $z_0 \in \lie{H}$ denote the point such that $J_{z_0} = - \phi(\mathbb I_k)$, 
so that the orthogonal complement (with respect to the reduced norm form) is the negative plane
\[ J_{z_0}^{\perp} \ = \ \theta \phi(k_{\R}) \ \subset \ B^o_{\R}  = \{ v \in B_{\R} \ | \ Trd(b) = 0 \} .\]
Applying the definition of $R^o(x,z)$, cf. \eqref{RoDefEqn}, we have
\[ R^o(x, z_0) + 2 Nrd(x) \ = \ \frac{ 2 t^2|\Delta|}{Nrd(y)^2} \left( 4 N(a) N(b) \theta^2 + Nrd(y)^2 \right).\]
On the other hand, the negative $k_{\R}$-complex line $\zeta_0 \in \D(B^{\phi}_{\R})^-$ corresponding to $z_0$ is given by 
\[ \zeta_0 \ = \ \left\{ v \in B_{\R} \ | \ v J_{z_0}  =  - \phi(\mathbb I_k) v \right\} \ = \ \phi(k_{\R}), \]
so 
\[ R_{\phi}(y, z) + \Delta Nrd(y) = 2 |\Delta| N(a) + \Delta Nrd(y). \]
This immediately implies that
\[ R^o(x, z_0) \ + \ 2 Nrd(x) \ = \ \frac{2 t^2}{|\Delta| Nrd(y)^2} \ \left( R_{\phi}(y, z_0) \ + \ \Delta Nrd(y) \right)^2; \]
i.e. the lemma holds at $z_0$. 
Next, we observe that for any $\gamma \in B_{\R}^{\times, 1} \simeq SL_2(\R)$, we have
\[ R^o(x, \gamma \cdot z_0 ) \ = \ R^o( \gamma^{-1} x \gamma , z_0)  \qquad \text{and} \qquad R_{\phi}(y, \gamma \cdot z_0 ) = R_{\phi} (y \gamma, z_0),\]
and so by translation, the lemma holds for all $z \in \lie{H}$. 

\end{proof}

\section{Proof of main theorem: The Hodge component} \label{hodgeSec}
%
We begin by recalling the following formula of Bost for the height pairing. Suppose 
\[ \hat{\mathcal L} \  =  \ (\mathcal L, ( ||\cdot||_{\sigma_i})_{i=0,1}) \ \in \ \widehat{Pic}(\CBok)\]
 is a metrized line bundle, and denote its image in $\CHR(\CBok)$ by the same symbol. Then the pairing against an arithmetic divisor $\widehat\Zed = (\Zed, Gr(\Zed)) \in \CHR(\CBok)) $can be expressed as 
\begin{align*}
 \la \widehat\Zed, \ \widehat\omega \ra \ =& \ h_{\widehat{\mathcal L}} (\Zed) \ + \ \frac12 \sum_{i=0}^1 \int_{\CB(\C^{\sigma_i})} Gr^{\sigma_i}(\Zed) \cdot c_1(\widehat\omega)^{\sigma_i} \\
 =& \  h_{\widehat{\mathcal L}} (\Zed)  \ + \ \int_{\CB(\C^{\sigma_0})} Gr^{\sigma_0}(\Zed) \cdot c_1(\widehat\omega)^{\sigma_0}
 \end{align*}
where $h_{\widehat{\mathcal L }}(\Zed)$ is the \emph{arithmetic height of $\widehat{\mathcal L}$ along $\Zed$}, as in \cite[Section 3.2.2]{Bost}; note that $h_{\widehat{\mathcal L }}(\Zed)$ depends only on the cycle $\Zed$, and not its Green function. Thus for $\tau = u + iv \in \lie{H}$,
\[ \la \widehat\Phi_{/o_k}^o (\tau),\widehat\omega \ra \  = \  \la \widehat\Zed{}^o(0,v), \widehat\omega \ra \ + \ \sum_{\substack{ n \in \Z \\ n \neq 0}} \ \big[ h_{\widehat\omega}( \Zed^o(n)_{/o_k}) + \kappa^o(n, v) \big] q_{\tau}^n, \]
where 
\[ \kappa^o(n,v) \ := \  \int_{\CB(\C^{\sigma_0})} Gr^{o, \sigma_0}(n, v)(z) \ c_1(\widehat\omega). \]

Similarly, for $w = \xi + i \eta \in \lie{H}$, the Hodge component of the unitary generating series is 
\begin{align}
\la \widehat\Phi{}^u(w), \widehat\omega \ra \ =& \ \la\widehat\Zed(0, \eta), \widehat\omega \ra \ + \frac{1}{4h(k)}\sum_{\ell \neq 0} \sum_{[\phi]} \Big\la  \widehat\Zed(m, \phi, \eta) +\widehat\Zed{}^*(m, \phi, \eta) , \ \widehat\omega \Big\ra  \ q_w^{\ell} \notag \\
 =& \ \la\widehat\Zed(0, \eta), \widehat\omega \ra+ \sum_{\ell \neq 0} (H^u(\ell,\eta) \ + \ K^u( \ell , \eta))  \ e^{2 \pi i \ell \xi}, \label{uHodgeFExpEqn}
 \end{align}
where 
\[ H^u(\ell,\eta) \ := \frac{1}{4h(k)} \sum_{\substack{[\phi] \in Opt \\ \text{mod } \OBxone}} \ \left[ h_{\widehat\omega} \Zed(\ell, \phi)  + h_{\widehat\omega} \Zed^*(\ell, \phi) \right] \ e^{-2 \pi \ell \eta} \]
and (using  \refLemma{genFibreUnitCorLemma}) 
\begin{align}
K^u(\ell, \eta) \ :=& \ \frac{1}{4h(k)} \Big( \int_{\CB(\C^{\sigma_0})} \sum_{\substack{[\phi] \in Opt \\ \text{mod } \OBxone}}  \left( Gr^{u, \sigma_0}(\ell, \phi, \eta) + Gr^{u*,\sigma_0}(\ell, \phi, \eta) \right) c_1(\widehat\omega) \Big) e^{-2 \pi \ell \eta} \notag \\
=&  \ \frac{1}{2h(k)} \Big( \int_{\CB(\C^{\sigma_0})} \sum_{\substack{[\phi] \in Opt \\ \text{mod } \OBxone}}  Gr^{u, \sigma_0}(\ell, \phi, \eta) \  c_1(\widehat\omega) \Big) e^{-2 \pi \ell \eta}.  \label{KuDefEqn}
\end{align} 

\begin{theorem} Assume $|\Delta|$ is squarefree and even and every prime dividing $D_B$ is inert in $k$. Then we have an equality of $q$-expansions: 
\[ Sh_{|\Delta|} \ \la \widehat{\Phi}_{/o_k}^o(\cdot), \widehat\omega \ra  \ (w)  \ = \ \la \widehat\Phi{}^u(w), \ \widehat\omega \ra. \]

\begin{proof} The proof begins in the same manner as the proof of  \refThm{anCompThetaThm}: the generating series $\la \widehat\Phi_{/o_k}^o, \widehat\omega \ra$ is a non-holomorphic modular form of weight 3/2, level $4 D_B$, and trivial character, and so its Shimura lift is given by integrating against the Niwa-Shintani theta kernel $\theta^{\#}_{NS}$.
We wish to substitute the Poincar\'e series expansion for $\theta^{\#}_{NS}$, as in \refThm{thetaPoincare}, and apply the usual unfolding trick. Doing so formally for the moment, we have
\begin{align}
Sh_{|\Delta|} \widehat\Phi{}^o_{\widehat\omega}(w) \ =& \  \frac{1}{8}  \int_0^{\infty} \int_{0}^1  v^{-1/2} \ \widehat\Phi{}^o_{\widehat\omega}(|\Delta|\tau)  \Bigg[ \sum_{m \in \Z} \chi'(m) \sum_{\mu = 0}^1 4^{\mu} \left( \frac{\eta m}{v} \right)^{1-\mu} \notag \\ 
& \qquad \times \exp \left( - \frac{\pi \eta^2 m^2}{4v} \right) \overline{\theta_{\mu}(\tau, - \xi m /2)} \Bigg]  du \ dv  \label{HodgeUnfoldEqn}
\end{align}
where $C = -2(|\Delta|D_B)^{-3/4}$ , the measure $d \mu(\tau) = du \cdot dv / v^2$ is the standard hyperbolic measure on $\lie{H}$, and 
\[ \theta_{\mu}(\tau, - \xi m/2) = \sum_{\ell \in \Z} \ell^{\mu} e^{2 \pi i \left( \tau \ell^2 - \xi \ell m \right)}. \]
We have also used the abbreviation $\widehat{\Phi}{}^o_{\widehat\omega}(\tau) = \la \widehat\Phi_{/o_k}^o(\tau), \widehat\omega \ra$ for the Hodge component. 

The Fourier coefficients of $\widehat\Phi{}^o_{\widehat\omega}(\tau)$ were computed completely explicitly by Kudla, Rapoport and Yang, cf.\ \cite[ Theorem 8.8]{KRYcomp}. From their calculations, one can show that \eqref{HodgeUnfoldEqn} is absolutely convergent, and so the unfolding procedure above (as well as interchanges of summations and integrations appearing in the sequel) are justified.
Carrying out the integral on $u$, we obtain
\begin{align*}
 \int_0^1  \widehat\Phi{}^o_{\widehat\omega}( & |\Delta|(u + i v)) \  \overline{\theta_{\mu}  (u + iv, - \xi m /2)} du  \\
 =& \sum_{\ell \in \Z} \ \ell^{\mu} \ e^{-2 \pi v \ell^2 + 2 \pi i \xi \ell m} \ \int_0^1 \widehat{\Phi}{}^o_{\widehat\omega}(|\Delta| (u + i v)) \ e^{-2 \pi i \ell^2 u} \ du \\
  =& \sum_{\ell \in \Z} \ \ell^{\mu} \ e^{-4 \pi v \ell^2 + 2 \pi i \xi \ell m} \  
  		\begin{cases} \la \widehat\Zed{}^o(0,v), \widehat\omega \ra, & \text{if } \ell = 0, \\
		h_{\widehat\omega}(\Zed^o(|\Delta|n))+ \kappa^o(|\Delta|n, |\Delta|v) & \text{if } \ell^2 = |\Delta| n \text{ for some } n > 0 \\ 
		 0, & \text{otherwise}. \end{cases}
\end{align*}
Note that only terms with $|\Delta|$ dividing $\ell$ contribute above, and that the constant term only appears when $\mu = 0$. After an interchange of summation with integration, we obtain
\begin{equation} \label{oHodgeFExpEqn}
Sh_{|\Delta|} \widehat\Phi^o_{\widehat\omega}(w) = A^o( \eta)\ + \  \sum_{\substack{\ell \in \Z \\ \ell \neq 0}} \ \left( \sum_{0< m \mid  |\ell|} \chi'(m)  \cdot  \left( H^o(\ell, m,\eta) + K^o(\ell, m, \eta) \right) \right) \cdot e^{2 \pi i \ell |\Delta| \xi},
\end{equation}
where 
\[ A^o( \eta) \ := \ \frac{1}{8} \int_0^{\infty} v^{-1/2} \ \sum_{m \in \Z} \chi'(m) \ \exp \left( - \frac{  \pi \eta^2 m^2}{4v} \right)\ \left( \frac{\eta m}{v}  \right)\cdot  \la \widehat\Zed{}^o(0,v), \widehat\omega \ra \ dv, \]
 
\begin{align}
H^o(\ell, m, \eta) \ := & \frac{1}{4}  \int_0^{\infty} v^{-1/2} \Big[  \sum_{\mu = 0}^1 \left(\frac{|\Delta|\ell}{m} \right)^{\mu} 4^{\mu} \left( \frac{\eta m}{v} \right)^{1-\mu}  \Big] \notag \\
 & \  \times  \exp \left( - \frac{\pi \eta^2 m^2}{4v}  - \frac{4 \pi v |\Delta|^2 \ell^2}{m^2}  \right) h_{\widehat\omega}(\Zed^o(|\Delta|\ell^2/m^2)) \ dv, \label{HoDefEqn}
 \end{align}
 and
\begin{align} K^o(\ell, m, \eta)\ :=& \ \frac{1}{4}  \int_0^{\infty} v^{-1/2} \Big[  \sum_{\mu = 0}^1 \left(\frac{|\Delta|\ell}{m} \right)^{\mu} 4^{\mu} \left( \frac{\eta m}{v} \right)^{1-\mu}  \Big] \notag \\
& \  \times  \exp \left( - \frac{\pi \eta^2 m^2}{4v}  - \frac{4 \pi v |\Delta|^2 \ell^2}{m^2}  \right) \kappa^o(|\Delta|\ell^2/m^2, |\Delta|v) dv. \label{KoDefEqn}
\end{align}
To prove the theorem, we show that the Fourier expansions \eqref{uHodgeFExpEqn} and \eqref{oHodgeFExpEqn} coincide. We do this over the next three subsections, by equating first the contributions from the `height terms' (the $H$'s), then the `archimedean' terms (the $K$'s) and in the final subsection, the constant terms. 

\subsection{Height contributions}
First, using the formulas \eqref{bessK-3/2} and \eqref{bessK-1/2} to calculate the integral on $v$ in \eqref{HoDefEqn}, we find
\begin{align*}
H^o(\ell, m,\eta) \ =& \  \frac{1}{4} \ h_{\widehat\omega}( \Zed^o(|\Delta| \ell^2 / m^2 ) \ \cdot \  2 \ ( 1 + sgn(\ell))\ e^{- 2 \pi |\Delta| |\ell| \eta}  \\
=& \ \begin{cases} h_{\widehat\omega}( \Zed^o(|\Delta| \ell^2 / m^2 ) \cdot e^{- 2 \pi |\Delta| |\ell| \eta} , & \text{if } \ell > 0 \\ 0, & \text{otherwise.} \end{cases}
\end{align*}
By \cite[Theorem 4.10]{San1}, we have that if $|\Delta | \nmid \ell$, then
\[\sum_{[\phi] \in Opt / \OBxone} \Zed(\ell, \phi) + \Zed^*(\ell, \phi) \ = \ 0, \]
and if $\ell = |\Delta| \ell'$, then
\begin{equation} \label{cyclesRelEqn}
 \frac{1}{4h(k)} \sum_{[\phi] \in Opt / \OBxone}  \Zed(\ell, \phi) + \Zed^*(\ell, \phi) \ = \  \sum_{\alpha|\ell'} \chi'(\alpha)  \Zed^o \left( |\Delta| \frac{(\ell')^2}{\alpha^2} \right)_{/o_k};
\end{equation}
note that the quantity on the right-hand-side is twice that appearing in \cite{San1}, since here we have taken the sum on optimal embeddings modulo $\OBxone$, instead of $\OBx$. 

Thus, if $\ell <0$ or $|\Delta| \nmid \ell$, then $H^u(\ell, \eta) = 0$. On the other hand, for any $\ell >0$, we have
\begin{align*}
H^u(\ell|\Delta|, \eta)  \ =& \  \frac{1}{4h(k)} \sum_{\substack{[\phi] \in Opt \\ \text{mod } \OBxone}} \ \left[ h_{\widehat\omega} \Zed(\ell|\Delta|, \phi)  + h_{\widehat\omega} \Zed^*(\ell|\Delta|, \phi) \right] \ e^{-2 \pi \ell |\Delta| \eta} \\
=& \   \sum_{\alpha|\ell} \chi'(\alpha) \  h_{\widehat\omega}  \Zed^o \left( |\Delta| \ell^2 / \alpha^2 \right) \cdot  e^{- 2 \pi \ell |\Delta|  \eta} \ = \ \sum_{\alpha|\ell} H^o(\ell, \alpha,\eta),
\end{align*}
as required. 
\subsection{Archimedean contributions} {\ \\}

In \cite[Proposition 12.1]{KRYcomp}, the quantity $\kappa^o(|\Delta|\ell^2 / m^2,|\Delta| v )$ appearing in \eqref{KoDefEqn} above is computed to be
\begin{align*} 
\kappa^o \left( |\Delta| \ell^2/m^2 ,|\Delta| v\right) &=  \sum_{\substack{ x \in \Omega^o(|\Delta|\ell^2/m^2) \\ mod \OBxone}} \frac{1}{|\Gamma_x|}  \cdot \int_0^{\infty} e^{-4 \pi |\Delta|^2 (\ell/m)^2 v s} \left( (s+1)^{1/2} - 1 \right) \frac{ds}{s} \\
&= 2 \sum_{\substack{ x \in \Omega^o(|\Delta|\ell^2/m^2) \\ mod \OBxone}} \frac{1}{|\Gamma_x|}  \cdot \int_1^{\infty} \exp \Big[ -4 \pi |\Delta|^2 (\ell/m)^2 v (s^2 -1) \Big] \frac{s }{s+1} ds,
\end{align*}
where $\Gamma_x = \{ \gamma \in \OBxone \ | \ \gamma x = x \gamma \}$. Note that the right hand side of this formula is twice the value appearing in \emph{loc.\ cit.} because  here we have taken the base change to $o_k$; we have also written our formula in terms of $\OBxone$-orbits, rather than $\OBx$. 

Substituting this expression into  \eqref{KoDefEqn} and changing the order of integration, we find
\[ K^o(\ell, m, \eta) \ = \ \frac12 \ \Big( \sum_{ \substack{ x \in \Omega^o(|\Delta|\ell^2/m^2) \\ \text{mod } \OBxone}} |\Gamma_x|^{-1} \Big)\ I^o(m, \ell, \eta),\]
where 
\begin{align*}
 I^o(m, \ell, \eta) :=& \ \int_1^{\infty} \ \int_0^{\infty} v^{-1/2}  \Big[  \sum_{\mu = 0}^1 \left(\frac{|\Delta|\ell}{m} \right)^{\mu} 4^{\mu} \left( \frac{\eta m}{v} \right)^{1-\mu}  \Big]  \\
 & \qquad \times \exp \left( - \frac{\pi\eta^2m^2}{4v} -  \frac{4 \pi v |\Delta|^2 \ell^2 s^2}{m^2}  \right)  dv \ \cdot \left( \frac{s}{s+1} \right) ds.
\end{align*}
The inner integral, which turns out to be independent of $m$, can be explicitly evaluated with the aid of \eqref{bessK-3/2} and \eqref{bessK-1/2}, yielding
\begin{equation*} 
 I^o(m, \ell, \eta) = I^o(\ell, \eta) \ = \int_1^{\infty} 2 \left( 1 + \frac{sgn(\ell)}{s} \right) e^{-2\pi |\Delta| \eta |\ell| s}  \frac{s}{s+1} ds.
\end{equation*}
From this, a straightforward calculation shows that 
\begin{equation} \label{IoEllPosEqn}
 I^o(\ell, \eta) = \frac{1}{\pi \ell |\Delta| \eta} \  e^{ - 2 \pi |\Delta| \ell \eta} \qquad \text{if } \ell > 0, 
\end{equation}
and 
\begin{equation} \label{IoEllNegEqn}
I^o(\ell, \eta) \ = \ \frac{1}{\pi | \ell\Delta | \eta} \ e^{ - 2 \pi | \ell\Delta | \eta}  \ - \ 4 e^{2 \pi |\Delta \ell| \eta} \ \int_1^{\infty} e^{-4 \pi |\ell \Delta| \eta s } \frac{ds}{s}, \qquad \text{ if } \ell < 0. \end{equation}
In summary, we have that the archimedean part of the $|\Delta|\ell$'th Fourier coefficient of $Sh_{|\Delta|} \Phi^o_{\widehat\omega}$ is 
\begin{align*}
\sum_{m | |\ell|} \chi'(m) K^o(\ell, m, \eta)  \ =& \ \frac12 \ \left( \sum_{m | |\ell|} \chi'(m) \sum_{ \substack{ x \in \Omega^o(|\Delta|\ell^2/m^2) \\ \text{mod } \OBxone}} |\Gamma_x|^{-1} \right) \ I^o(\ell, \eta).
\end{align*} 

We turn to the corresponding term  $K^u(\ell, \eta)$ for the unitary series, cf.\ \eqref{KuDefEqn}. If $|\Delta|$ does not divide $\ell$, then $\Omega^{\pm}(\ell, \lie{a}, \phi)=\emptyset$ and so $K^u(\ell, \eta) = 0$ in this case. On the other hand, since a generic point in $\lie{H}$ has a stabilizer of order 2 under the action of $\OBxone$, we have 
\[ \int_{\CB(\C^{\sigma_0}) }  \ = \ \int_{ [\OBxone \backslash \lie{H} ] } \ = \ \frac12 \int_{\OBxone \backslash \lie{H} },  \]
and so, applying the definitions for the Green functions,
\begin{align}
K^u(|\Delta| \ell, \eta) \ =&  \  \frac{e^{-2 \pi \ell |\Delta| \eta}}{ 4h(k) |o_k^{\times}|} \cdot \sum_{[\phi]} \   \int_{\OBxone \backslash \lie{H} }  \sum_{[\lie{a}]} \Big[  \sum_{y \in \Omega^+(|\Delta|\ell, \lie{a}, \phi)}  \beta_1\left( 2 \pi N(\lie{a}) \eta R_{\phi}(y, z) \right) \notag \\
& \qquad + \sum_{y \in \Omega^-(|\Delta|\ell, \lie{a}, \phi)}  \beta_1\big( 2 \pi N(\lie{a}) \eta (R_{\phi}(y, z) + 2 \Delta Nrd(y)) \big) \ \Big] \cdot c_1(\widehat\omega)   \notag\\
=& \ \frac{e^{-2 \pi \ell |\Delta| \eta}}{ 2 h(k) |o_k^{\times}|} \cdot \Bigg[ \sum_{[\phi]} \ \sum_{[\lie{a}]}    \sum_{\substack{ y \in \Omega^{+}(|\Delta|\ell, \lie{a}, \phi) \\ \text{mod } \OBxone} }  \int_{ \lie{H} }   \beta_1\left( 2 \pi N(\lie{a}) \eta R_{\phi}(y, z) \right)  \cdot c_1(\widehat\omega) \notag\\
& \qquad +  \sum_{\substack{ y \in \Omega^{-}(|\Delta|\ell, \lie{a}, \phi) \\ \text{mod } \OBxone} }   \int_{\lie{H}} \beta_1\big( 2 \pi N(\lie{a}) \eta (R_{\phi}(y, z) + 2 \Delta Nrd(y)) \big) \cdot c_1(\widehat\omega) \Bigg] \label{KuExpEqn}  \\
=&\ \frac{1}{ 2 h(k)|o_k^{\times}|} \cdot \sum_{[\phi]}  \sum_{[\lie{a}]}  \sum_{\substack{ y \in \Omega^{\pm}(|\Delta|\ell, \lie{a}, \phi) \\ \text{mod } \OBxone} } I_{\phi}(y, \eta), \notag
\end{align}
where
\[ I_{\phi}(y, \eta) \ = \ e^{-2 \pi \ell |\Delta|\eta} \int_{\lie{H}} \beta_1\Big( 2 \pi N(\lie{a}) \eta R_{\phi}(y, z)\Big) \cdot c_1(\widehat\omega), \qquad \text{ if } y \in \Omega^+(|\Delta|\ell, \lie{a}, \phi),  \]
and
\[ I_{\phi}(y, \eta) = e^{-2 \pi \ell |\Delta|\eta} \int_{\lie{H}} \beta_1 \Big( 2 \pi N(\lie{a}) \eta (R_{\phi}(y, z) + 2 \Delta Nrd(y)) \Big)  \cdot c_1(\widehat\omega), \qquad \text{ if } y \in \Omega^-(|\Delta|\ell, \lie{a}, \phi). \]
Note that a factor of 2 emerges in \eqref{KuExpEqn} between the first and second lines because $\{ \pm 1 \}$ acts trivially on $\lie{H}$ but non-trivially on $\Omega^{\pm}(\dots)$. 

Fix an embedding $\phi \colon o_k \to \OB$, and an element $y \in \Omega^+(|\Delta| \ell, \lie{a}, \phi)$. Recall that $\phi$ induces an isomorphism 
\[ \lie{H} \simeq \D(B^{\phi}_{\R})^- \] 
as in \eqref{altCpxUnifMapEqn}.
We may fix an orthogonal $k_{\R}$ basis $\{ e, f \}$ of $B^{\phi}_{R}$ such that $(e, e)_{\phi} = -(f, f)_{\phi}=1$, which yields an isomorphism 
\[ U^1  \ = \ \{ \lie{Z} \in \C \ | \ |\lie{Z}| < 1 \} \ \isomto \D(B^{\phi}_{\R})^-, \qquad \lie{Z} \mapsto \text{span}_{k_{\R}}( \lie{Z} e + f) , \]
such that with respect to the parameter $\lie{Z} = r e^{i \theta} $ on $U^1$, we have
\[ c_1(\widehat\omega) \ = \frac{2}{\pi} \ \frac{r \ dr \wedge d \theta}{(1 -r^2)^2}, \]
cf.\  \refLemma{hodgeUnifLemma}. 
Thus, we may write
\[ I_{\phi}(y, \eta) \ = \frac{2e^{-2 \pi \ell |\Delta|\eta}}{\pi} \int_{U^1} \beta_1\Big( 2 \pi N(\lie{a}) \eta \cdot R_{\phi}(y, r e^{i \theta})\Big) \  \frac{r \ dr \ d \theta}{(1 -r^2)^2}.  \]
Since $U(B_{\R})$ acts transitively on $\D(B^{\phi}_{\R})^- \simeq U^1$, and the measure appearing in the above display is invariant for this action, we may assume without loss of generality that 
\[ y \ = \ \left( \frac{ |\Delta \ell|}{N(\lie{a})} \right)^{1/2} \times \begin{cases} e, & \text{if } \ell > 0 \\ f, & \text{if } \ell < 0 \end{cases}. \]
Following definitions, we obtain
\[2 \pi \eta N(\lie{a}) \cdot R_{\phi}(y, r e^{i\theta}) \ = \ - \frac{4 \pi |\ell \Delta| \eta}{ 1 - r^2} \times \begin{cases} r^2, & \text{if } \ell > 0 \\ 1, & \text{if } \ell < 0. \end{cases}\]
Thus, when $\ell > 0$, we have
\begin{align*}
 I_{\phi}(y, \eta) \ =& \ \frac{2e^{-2 \pi \ell |\Delta|\eta}}{\pi} \int_{U^1} \ \int_1^{\infty} \exp\left( - \frac{ 4 \pi |\ell\Delta| r^2 s }{1 - r^2} \right) \frac{ds}{s} \ \frac{r \ dr \ d\theta}{(1-r^2)^2} \\
 =& \ 4e^{-2 \pi \ell |\Delta|\eta} \int_0^1 \int_1^{\infty} \exp \left( - \frac{4 \pi |\ell\Delta| \eta r^2 s}{ 1 - r^2} \right) \frac{ds}{s}\frac{r \ dr }{(1-r^2)^2}  \\
 =& \ 4 e^{-2 \pi \ell |\Delta|\eta}   \int_1^{\infty}  \exp \left( 4 \pi |\ell\Delta| \eta  s \right) \left[ \int_0^1  \exp \left( - \frac{4 \pi |\ell\Delta| \eta  s}{ 1 - r^2} \right) \frac{r \ dr }{(1-r^2)^2}  \right] \frac{ds}{s}  \\
 =& \ \frac{e^{-2 \pi \ell |\Delta|\eta} }{2\pi|\ell\Delta| \eta} \int_1^{\infty} \frac{ds}{s^2} \ = \ \frac{e^{-2 \pi \ell |\Delta|\eta} }{2\pi\ell |\Delta| \eta}.
 \end{align*}
 On the other hand, if $\ell < 0 $, we obtain
 \begin{align*}
  I_{\phi}(y, \eta) \ =& \frac{2 e^{-2 \pi \ell |\Delta|\eta} }{\pi} \int_{U^1} \ \int_1^{\infty} \exp\left( - \frac{ 4 \pi |\ell\Delta|\eta s }{1 - r^2} \right) \frac{ds}{s} \ \frac{r \ dr \ d\theta}{(1-r^2)^2} \\
  =& \ 4 \ e^{-2 \pi \ell |\Delta|\eta}   \int_1^{\infty} \frac{ds}{s}  \cdot \int_0^1  \exp\left( - \frac{ 4 \pi |\ell\Delta| \eta s }{1 - r^2} \right) \frac{r \ dr}{(1-r^2)^2} \\
  =& \ \frac{e^{-2 \pi \ell |\Delta|\eta}}{2\pi |\ell\Delta|\eta} \int_1^{\infty} e^{-4 \pi |\ell\Delta| \eta s} \ \frac{ds}{s^2} \\
  =& \ e^{-2 \pi \ell |\Delta|\eta} \cdot \left( \frac{1}{2\pi |\ell\Delta|\eta}  \ e^{-4 \pi |\ell\Delta| \eta}  \ - 2 \int_{1}^{\infty} e^{- 4 \pi |\ell\Delta| \eta s} \frac{ds}{s} \right)  \\
  =& \ \frac{1}{2\pi |\ell\Delta|\eta}  \ e^{-2 \pi |\ell\Delta| \eta}  \ - \ 2 e^{2 \pi |\ell\Delta| \eta} \int_{1}^{\infty} e^{- 4 \pi |\ell\Delta| \eta s} \frac{ds}{s}  
 \end{align*} 
Comparing with \eqref{IoEllPosEqn} and \eqref{IoEllNegEqn}, we find that for all $\ell \neq 0$ and $y \in \Omega^+(|\Delta|\ell, \lie{a}, \phi)$,
\[ I_{\phi}(y, \eta) = \frac12 \ I^o(\ell, \eta). \] 
A similar calculation shows that the preceding display holds true for $y \in \Omega^-(|\Delta|\ell, \lie{a}, \phi)$. 
Thus, 
\[ K^u(|\Delta|\ell, \eta) \ = \ \frac{1}{4 h(k) |o_k^{\times}|}  \cdot \left( \sum_{[\phi]} \sum_{[\lie{a}]} \sum_{\substack{ y \in \Omega^{\pm}(|\Delta| \ell, \lie{a}, \phi) \\ \text{mod } \OBxone} } 1 \right)\ \cdot \left( \frac12 I^o(\ell, \eta) \right) . \]
By  \refProp{OmegasCompProp} and passing to $\OBxone$-equivalence classes, we have
\begin{align*}
\frac{1}{2 h(k)|o_k^{\times}|}\ \sum_{[\phi]} \ \sum_{[\lie{a}]} \sum_{\substack{ y \in \Omega^{\pm}(|\Delta| \ell, \lie{a}, \phi) \\ \text{mod } \OBxone} } 1 \ =& \  \frac{1}{2 h(k)|o_k^{\times}|} \ \sum_{[\phi]} \ \sum_{ \substack{ x \in \Omega^o(|\Delta| \ell^2) \\ \text{mod } \OBxone}} \rho\left( \frac{|\ell|}{ c(x) \cdot \nu(x, \phi) } \right)  \\
=& \ \left( \sum_{c | |\ell|} \ \sum_{\nu | D_B} \rho\left( \frac{|\ell|}{ c \cdot \nu } \right)  \right) \cdot \Bigg( \frac{1}{2 h(k)} \ \sum_{[\phi]}  \sum_{ \substack{ x \in \Omega^o(|\Delta| \ell^2) \\ \text{mod } \OBxone \\ c(x) = c, \ \nu(x, \phi) = \nu}} \frac{1}{|o_k^{\times}|}  \Bigg) \\
=&  \ \sum_{m | |\ell|} \chi'(m) \ \sum_{ \substack{ x \in \Omega^o(|\Delta| \ell^2 / c^2) \\ \text{mod } \OBxone }} \frac{1}{|\Gamma_x|},
\end{align*}
where in the last line, we have used  \refLemma{rhoLemma}, together with the fact that $\Gamma_x = |o_k^{\times}|$ for the $x$'s appearing above. 
Therefore, we have 
\[ K^u(|\Delta|\ell, \eta) \ = \frac12 \Big( \sum_{m | |\ell|} \chi'(m) \sum_{ \substack{ x \in \Omega^o(|\Delta|\ell^2/m^2) \\ \text{mod } \OBxone}} \frac{1}{|\Gamma_x|} \Big) \ I^o(\ell, \eta) \ = \ \sum_{m||\ell|} \chi'(m) \  K^o(m, \ell, \eta), \]
as required. 

\subsection{The constant term}
Recall that we had computed the constant term of $Sh_{|\Delta|}\widehat\Phi{}^o_{\widehat\omega}$ to be 
\begin{align*}
A^o( \eta) \ :=&  \ \frac{1}{8} \int_0^{\infty} v^{-1/2} \ \sum_{m \in \Z} \chi'(m) \ \exp \left( - \frac{  \pi \eta^2 m^2}{4v} \right)\ \left( \frac{\eta m}{v}  \right) a^o(|\Delta|v) \ dv \\ 
=& \ \frac14 \int_0^{\infty} \sum_{m \in \Z} \chi'(m) \ \exp \left( - \frac{  \pi \eta^2 m^2v^2}{4} \right)\ \left( \eta m \right) a^o(|\Delta|/v^2) \ dv
\end{align*}
where 
\begin{align*}
 a^o(v) :=& \  \la \widehat{\Zed}{}^o(0,v)_{/o_k}, \widehat\omega \ra  =- \la \widehat\omega, \widehat\omega \ra   -   \deg \omega_k \left( \log v + \log D_B \right) 
 \end{align*}
is the constant term of $\widehat\Phi{}^o_{\widehat\omega}(|\Delta| \tau)$.

To help our formulas appear somewhat neater, let $D = D_B$ and $t = |\Delta|$.

Applying twisted Poisson summation to the inner sum on $m$, cf.\ \cite[(2.42)]{Cipra}, yields:
\begin{align*}
 \sum_{m \in \Z} \chi'(m) \ &m \ \exp \left( - \frac{\pi \eta^2 m^2 v^2}{4} \right)  \\
&= \frac{-i}{2D^2t^2} \cdot \frac{1}{\eta^3 v^3} \left[ \sum_{m \in \Z} \check{\chi}'(m) m \exp \left( - \frac{ \pi m^2}{4D^2 t^2 \eta^2 v^2} \right) \right],
 \end{align*}
 where 
 \[ \check{\chi}'(m) = \sum_{a =1}^{4Dt} \overline{\chi}'(a) \ e(am/4Dt). \]

Substituting and applying a change of variables yields
\begin{align*}
 A^o(\eta) &= \frac{-i}{8D^2t^2} \int_0^{\infty}  a^o( t/v^2) \cdot \frac{1}{\eta^2 v^3} \sum_{m \in \Z} \check{\chi}'(m) m \exp \left( \frac{- \pi m^2}{4D^2 t^2 \eta^2 v^2} \right) dv \\ 
&= \frac{-i}{2} \int_0^{\infty} \sum_{\substack{ m \in \Z \\ m > 0}}a^o( 4D^2 t^3 \eta^2 u / m^2 ) \cdot \check{\chi}'(m) \ m^{-1} \ \exp( - \pi u) \ du \\
&= \frac{i}{2} \ \int_0^{\infty} \sum_{m>0}  \check{\chi}'(m) \ m^{-1} \Big( \la \widehat\omega, \widehat\omega \ra  +  \deg (\omega_k) [ \log u + \log (4 D^3t^3 \eta^2) - 2 \log m) \Big) \ e^{- \pi u} du. 
\end{align*} 
Expanding, we obtain
\begin{align}
A^o(\eta) \ =& \ \frac{i}{2}  \Bigg[ \ L(1,  \check{\chi}')  \cdot \left(  \la \widehat\omega, \widehat\omega \ra  + \deg(\omega_k) \log 4 D^3 t^3 \eta^2 \right) \cdot \left( \int_{0}^{\infty} e^{-\pi u } \ du \right) \notag \\
& \ \ +  L(1,  \check{\chi}') \deg(\omega_k) \left( \int_{0}^{\infty} \log(u) e^{-\pi u } \ du \right)  \ - \ 2 L'(1, \check{\chi}') \cdot  \deg(\omega_k)\left( \int_{0}^{\infty} e^{-\pi u } \ du \right) \Bigg] \notag \\
=& \ \frac{i}{2\pi}  \Bigg[ \ L(1,  \check{\chi}')  \cdot \left(  \la \widehat\omega, \widehat\omega \ra  +  \deg(\omega_k)(\log 4 D^3 t^3 \eta^2 - \gamma - \log \pi) \right) - 2 L'(1, \check{\chi}') \cdot  \deg(\omega_k) \Bigg], \label{constTermHodge}
\end{align}
where we have used the integral representation
\[ \gamma = - \int_0^{\infty} \log (u) e^{-u} \ du. \]
Now comparing \eqref{constTermHodge} with  \refDef{uConstDef} for the constant term $\widehat\Zed(0, \eta)$, we immediately see that
\[ A^o(\eta) \ = \ \la \widehat\Zed(0,\eta) , \ \widehat\omega \ra, \]
as required. 
\end{proof}

\end{theorem}

\end{document}